\documentclass[12pt,a4paper,psamsfonts]{amsart}

\usepackage{fancyhdr}
\usepackage{appendix}
\usepackage{amssymb,amscd,amsxtra,calc}
\usepackage{mathrsfs}
\usepackage{mathtools}
\usepackage{cmmib57}
\usepackage{multirow}
\usepackage[all]{xy}
\usepackage{longtable}
\usepackage[colorlinks,linkcolor=blue,anchorcolor=blue,citecolor=green]{hyperref}

\usepackage{mathabx,epsfig}
\def\acts{\ \rotatebox[origin=c]{-90}{$\circlearrowright$}\ }
\def\racts{\ \rotatebox[origin=c]{90}{$\circlearrowleft$}\ }

\theoremstyle{plain}
    \newtheorem{thm}{Theorem}[section]
    
    \newtheorem{claim}[thm]{Claim}
     \newtheorem{conjecture}[thm]{Conjecture}
    \newtheorem{corollary}[thm]{Corollary}
    \newtheorem{lemma}[thm]{Lemma}
    \newtheorem{proposition}[thm]{Proposition}
    \newtheorem{question}[thm]{Question}
    \newtheorem{theorem}[thm]{Theorem}

\theoremstyle{definition}
    \newtheorem{example}[thm]{Example}
    \newtheorem{definition}[thm]{Definition}
    
    \newtheorem*{notation*}{Notation and Terminology}
    \newtheorem*{convention*}{Convention}
    \newtheorem{remark}[thm]{Remark}
\theoremstyle{remark}

\newcommand{\C}{\mathbb{C}}

\newcommand{\Q}{\mathbb{Q}}

\newcommand{\R}{\mathbb{R}}

\newcommand{\Z}{\mathbb{Z}}

\newcommand{\Amp}{\operatorname{Amp}}

\newcommand{\Exc}{\operatorname{Exc}}

\newcommand{\id}{\operatorname{id}}

\newcommand{\NE}{\overline{\operatorname{NE}}}
\newcommand{\Nef}{\operatorname{Nef}}

\newcommand{\NS}{\operatorname{NS}}

\newcommand{\PE}{\operatorname{PE}}

\newcommand{\Sing}{\operatorname{Sing}}

\newcommand{\Supp}{\operatorname{Supp}}

\newcommand{\N}{\operatorname{N}}

\newcommand{\Sym}{\operatorname{Sym}}

\newcommand{\Pic}{\operatorname{Pic}}

\newcommand{\red}{\mathrm{red}}

\begin{document}

\title[Endomorphisms of projective threefolds]
{Structures theorems and applications of non-isomorphic surjective endomorphisms of smooth projective threefolds}

\author{Sheng Meng, De-Qi Zhang}

\address{
    \textsc{School of Mathematical Sciences, Ministry of Education Key Laboratory of Mathematics and Engineering Applications \& 
    Shanghai Key Laboratory of PMMP}\endgraf
    \textsc{East China Normal University, Shanghai 200241, China}\endgraf
}
\email{smeng@math.ecnu.edu.cn}

\address
{
\textsc{Department of Mathematics} \endgraf
\textsc{National University of Singapore,
Singapore 119076, Republic of Singapore
}}
\email{matzdq@nus.edu.sg}

\begin{abstract}
Let $f:X\to X$ be a non-isomorphic (i.e., $\deg f>1$) surjective endomorphism of a smooth projective threefold $X$.
We prove that any birational minimal model program becomes $f$-equivariant after iteration, provided that $f$ is  $\delta$-primitive. Here $\delta$-primitive means that there is no $f$-equivariant (after iteration) dominant rational map $\pi:X\dashrightarrow Y$ to a positive lower-dimensional projective variety $Y$ such that the first dynamical degree remains unchanged.
This way, we further determine the building blocks of $f$.

As the first application, we prove the Kawaguchi-Silverman conjecture for every non-isomorphic surjective endomorphism of a smooth projective threefold.
As the second application, we reduce the Zariski dense orbit conjecture for $f$ to a terminal threefold with only $f$-equivariant Fano contractions.
\end{abstract}



\maketitle
\tableofcontents

\section{Introduction}

Unless otherwise specified, we always assume that the base field $K$ is an algebraically closed field of characteristic $0$.

Let $f:X\to X$ be a non-isomorphic (i.e., $\deg f>1$) surjective endomorphism of a normal projective variety $X$.
It is fundamental to characterize $f$ and $X$.
A natural reduction idea is to ``decompose'' $f$ and $X$ into lower-dimensional objects.
A classical approach is to make use of various fibrations with universal properties such as the Iitaka fibration, special maximal rationally connected fibration, and the Albanese map.

A more delicate surgery is to perform the Equivariant Minimal Model Program (EMMP) on a mildly singular $X$ as illustrated by the following commutative diagram
$$\xymatrix{
X=X_1\ar@{-->}[r]^{\pi_1}\ar[d]_{f_1:=f^s} & X_2\ar@{-->}[r]^{\pi_2}\ar[d]^{f_2} &\cdots \ar@{-->}[r]^{\pi_{n-1}}&X_n\ar[d]^{f_n}\\
X=X_1\ar@{-->}[r]^{\pi_1} & X_2\ar@{-->}[r]^{\pi_2} &\cdots \ar@{-->}[r]^{\pi_{n-1}}&X_n
}$$
for some positive integer $s>0$. 
Here $\pi_i$ is either a divisorial contraction, a flip, or a Fano contraction of a $K_{X_i}$-negative extremal ray.
After iteration of $f_i$, a Fano contraction is always equivariant, while a divisorial contraction (resp.~flip) is equivariant if we can show the exceptional locus (resp.~flipping locus) is $f_i$-periodic, see \cite[Lemma 6.2]{MZ18}.

When $X$ is a normal projective surface, we can run EMMP due to the finiteness of negative curves by Nakayama \cite[Section 2]{Nak02} (for the smooth case) and the authors \cite[Section 5]{MZ22} (for the general case).
The detailed characterization is further provided in \cite[Theorem 1.1]{JXZ23}.

In higher dimensions, one can also run EMMP when $X$ admits a polarized or an int-amplified endomorphism due to the finiteness of contractible extremal rays, in increased generality, proved in \cite{MZ18, CMZ20, Men20, MZ20a, MZ20b}.
Of course, we do not expect every MMP to be EMMP in general and indeed there is an easy counter-example in dimension $3$ that breaks this illusion, see Example \ref{example-emmp}.
So it is important to explore the phenomenon when EMMP fails.

{\it In the paper, we shall mainly focus on smooth projective threefolds.}

Surprisingly, a new $f$-equivariant fibration, called {\it dynamical Iitaka fibration}, will emerge whenever the birational MMP starting from $X$ fails to be $f$-equivariant even after iteration.
Further study on this fibration reveals that it preserves the {\it first dynamical degree} 
$$\delta_f:=\lim\limits_{m\to+\infty}((f^m)^*H\cdot H^{\dim X-1})^{1/m},$$
where $H$ is an ample divisor of $X$ (cf.~Definition \ref{def-d1}).
Precisely, we show the following:

\begin{theorem}\label{mainthm-A}
Let $f:X\to X$ be a non-isomorphic surjective endomorphism of a smooth projective threefold $X$.
Then, after iteration of $f$, either one of the following holds:
\begin{enumerate}
\item any birational MMP of $X$ is $f$-equivariant, or
\item there is an $f$-equivariant dominant rational map $\phi:X\dashrightarrow Y$ to a normal projective variety $Y$ with $0<\dim Y<3$ such that the first dynamical degree is preserved: $\delta_f=\delta_{f|_Y}$ (equvalently, $f$ is $\delta$-imprimitive as in Definition \ref{def_endo_notion}).
\end{enumerate}
\end{theorem}

Now we introduce various notions of endomorphisms for the better exposition.

\begin{definition}\label{def_endo_notion} Let $f:X\to X$ be a surjective endomorphism of a normal projective variety $X$.
\begin{enumerate}
\item $f$ is {\it non-isomorphic} if $\deg f>1$.
\item $f$ is {\it imprimitive} if there exists an $f^s$-equivariant dominant rational map $\pi:X\dashrightarrow Y$ for some $s>0$, with $0<\dim Y<\dim X$.
\item $f$ is {\it primitive} if it is not imprimitive.
\item $f$ is {\it $\delta$-imprimitive} if there exists an $f^s$-equivariant dominant rational map $\pi:X\dashrightarrow Y$ for some $s>0$, with $\delta_{f^s|_Y}=\delta_{f^s}$  and $0<\dim Y<\dim X$.
\item $f$ is {\it $\delta$-primitive} if it is not $\delta$-imprimitive.
\item $f$ is {\it weakly $\delta$-imprimitive} if there exists a finite surjective morphism $\pi:\widetilde{X}\to X$ such that $f$ lifts to a surjective endomorphism $\widetilde{f}:\widetilde{X}\to \widetilde{X}$ and $\widetilde{f}$ is $\delta$-imprimitive.
\item $f$ is {\it strongly $\delta$-primitive} if it is not weakly $\delta$-imprimitive.
\item $f$ is {\it strongly imprimitive} if there is an $f^s$-equivariant dominant rational map $\pi:X\dashrightarrow Y$ for some $s>0$, with $\dim Y=1$ (equivalently $\dim Y>0$ or $Y=\mathbb{P}^1$) and $f^s|_Y=\id_Y$.
\item $f$ is {\it weakly primitive} if it is not strongly imprimitive.
\item $f$ is {\it quasi-abelian} if it is a surjective endomorphism of a quasi-abelian variety, i.e., there exists a finite surjective morphism $\pi:A\to X$ \'etale in codimension $1$ such that $f^s$ lifts to a surjective endomorphism $\widetilde{f}:A\to A$ of an abelian variety $A$ for some $s>0$.
\item $f$ is {\it $q$-polarized} (or simply polarized) if $f^*H\sim qH$ for some $q>1$ and ample Cartier divisor $H$. 
\end{enumerate}

With the above notion, Theorem \ref{mainthm-A} says that: starting from a smooth projective threefold, we can run any birational EMMP if $f$ is non-isomorphic and $\delta$-primitive.
A subsequent issue is whether a suitable choice of EMMP, especially the Fano contraction, preserves the first dynamical degree.
So we further explore how the EMMP behaves. 

When the canonical divisor $K_X$ is pseudo-effective (equivalently the Kodaira dimension $\kappa(X)\ge 0$ by the three-dimensional minimal model theory,
see \cite[\S 3.13]{KM98}), Fujimoto and Nakayama \cite{FN07} proved that any EMMP involves only smooth $X_i$ and ends up with a smooth minimal model $X_n$.
Furthermore, a full classification is given there and we can conclude that $f$ belongs to one of the three cases: weakly $\delta$-imprimitive, quasi-abelian, and strongly imprimitive.

The situation is much more complicated when $K_X$ is not pseudo-effective.
After finitely many steps of birational EMMP, we arrive at a Fano contraction
$\tau:X_r\to X_{r+1}=:Y$.
If $\dim Y=0$, then $f$ is polarized and one can explore more based on previous results.
Here, we do not go into details on the literature review but only mention one recent result by the first author \cite{Men22} that such $X$ is {\it log Calabi-Yau}, i.e., $(X,\Delta)$ is log canonical for some effective $\Q$-divisor $\Delta$ such that $K_X+\Delta\sim_{\Q} 0$.
Interestingly, the log Calabi-Yau and the log canonical pairs play an essential role in our treatment of the case $0<\dim Y<\dim X$.
In this way, we can show that $f$ is either weakly primitive, $\delta$-primitive, or polarized after iteration.

Now we state the structure theorem.

\end{definition}

\begin{theorem}\label{mainthm-B}
Let $f:X\to X$ be a non-isomorphic surjective endomorphism of a smooth projective threefold $X$.
Then the following hold in terms of the Kodaira dimension $\kappa(X)$.
\begin{enumerate}
\item If $\kappa(X)<0$, then $f$ is either $\delta$-imprimitive, strongly imprimitive, or polarized
after iteration.
\item If $\kappa(X)=0$, then $f$ is either weakly $\delta$-imprimitive or quasi-abelian.
\item If $\kappa(X)>0$, then $f$ is strongly imprimitive.
\end{enumerate}
\end{theorem}

Based on Theorem \ref{mainthm-B}, we propose the following conjecture.

\begin{conjecture}[Dream Building Blocks]\label{mainconj}
Let $f:X\to X$ be a non-isomorphic surjective endomorphism of a normal projective variety $X$.
Then $f$ is either weakly $\delta$-imprimitive, strongly imprimitive, quasi-abelian, or polarized after iteration.
\end{conjecture}

The Dream Building Blocks conjecture asserts:
$$\textbf{``Problems on endomorphisms should be reduced to tractable cases.''}$$
Conjecture \ref{mainconj} has been verified for surfaces (even with the possibility ``strongly imprimitive'' removed) 
by Theorem \ref{thm-dbbc-surf}, and smooth projective threefolds by Theorem \ref{mainthm-B}.

\vskip 2mm

In the rest of this section, we discuss several applications.

Let $f:X\to X$ be a surjective endomorphism of a projective variety $X$ over $\overline{\mathbb{Q}}$.
The {\it arithmetic degree} is defined as a function
$$\alpha_f(x) := \lim\limits_{m\to+\infty} \max\{1, h_H (f^m(x))\}^{1/m},$$
where $h_H$ is a Weil height function associated with an ample divisor $H$ of $X$; see Definition \ref{def-a-deg}.
The following {\it Kawaguchi - Silverman Conjecture} \ref{conj-KSC} ({\it KSC} for short, see \cite{KS16b}) asserts that $\alpha_f$ is well-defined, i.e., the limit exists, and the dynamical degree $\delta_f$ equals the arithmetic degree $\alpha_f(x)$ at any point $x$ with Zariski dense $f$-orbit.

\begin{conjecture}\label{conj-KSC} ({\bf Kawaguchi-Silverman Conjecture = KSC})
Let $f:X\to X$ be a surjective endomorphism of a projective variety $X$ over $\overline{\mathbb{Q}}$.
Then the following hold.
\begin{itemize}
\item[(1)] The limit defining arithmetic degree $\alpha_f(x)$ exists for any $x \in X(\overline{\mathbb{Q}})$.
\item[(2)] If the (forward) orbit $O_f(x) = \{f^n(x)\, |\, n\ge 0\}$ is Zariski dense
in $X$, then the arithmetic degree at $x$ is equal to the dynamical degree of $f$, i.e.,
$\alpha_f(x) = \delta_f$.
\end{itemize}
\end{conjecture}

The original conjecture is formulated for dominant rational self-maps of smooth projective varieties.
In our setting, Conjecture \ref{conj-KSC} (1) has been proved by Kawaguchi and Silverman themselves (cf.~\cite{KS16a});
more precisely, $\alpha_f(x)$ is either $1$ or the absolute value of an eigenvalue of $f^*|_{\NS(X)}$ for any $x\in X(\overline{\mathbb{Q}})$, see also \cite{Mat20b}.
In particular, $\alpha_f(x)\le \delta_f$.

We list the known results, to the best of our knowledge, towards KSC.
\begin{remark}\label{rmk-ksc}
The KSC holds for a surjective endomorphism $f$ 
on a projective variety $X$ which fits one of the following cases.
\begin{enumerate}
\item $f$ is polarized (cf. \cite{KS14}).
\item $f$ is quasi-abelian (cf. \cite{Sil17} and \cite[Theorem 2.8]{MZ22}).
\item $X$ is a Mori dream space (cf. \cite{Mat20a}).
\item $\dim X\le 2$ (for smooth case, see \cite{MSS18}; for singular case, see \cite{MZ22}).
\item $X$ is a Hyperk\"ahler manifold (cf. \cite{LS21}).
\item $X$ is a smooth rationally connected projective variety admitting an int-amplified endomorphism (see \cite{MZ22} and \cite{MY22} for respectively threefolds and $n$-folds).
\item $X$ is a projective threefold which admits an int-amplified endomorphism and has at worst $\Q$-factorial terminal singularities (cf. \cite[Corollary 6.19]{MMSZ23}).
\item $X$ is a smooth projective threefold with irregularity $q(X)>0$ and $f$ is an automorphism (cf. \cite{CLO22}).
\end{enumerate}
We also refer to \cite{CLO22, JSXZ21, KS14, MW22} for KSC in other settings.
\end{remark}

We briefly explain how Conjecture \ref{mainconj} and the automorphic KSC would imply KSC.
According to Remark \ref{rmk-ksc}, KSC holds for polarized and quasi-abelian cases.
Since KSC itself requires a Zariski dense orbit, we may assume $f$ is weakly primitive.
We may also assume $f$ is non-isomorphic if automorphic KSC holds.
Conjecture \ref{mainconj} then asserts that $f$ is $\delta$-imprimitive.
Thus KSC holds by the induction on $\dim X$ and Lemma \ref{lem-ksc-iff}.
In particular, applying the structure Theorem \ref{mainthm-B}, we have the following progress on KSC.

\begin{theorem}\label{mainthm-C}
Kawaguchi-Silverman conjecture holds for every non-isomorphic surjective endomorphism of a smooth projective threefold.
\end{theorem}

A stronger version of KSC, called the small Arithmetic Non-Density ($=$sAND) conjecture, also takes into account the points with no Zariski dense orbit, see \cite[Conjecture 1.3]{MMSZ23} or Conjecture \ref{conj-sand}.
The only troubled case is when $f$ is strongly imprimitive.
Excluding this, the same strategy of the proof of Theorem \ref{mainthm-C} works:  sAND holds for every weakly primitive and non-isomorphic surjective endomorphism of a smooth projective threefold, see Theorem \ref{thm-sand}.

Another related hard conjecture: Zariski dense orbit (ZDO for short) conjecture
asserts that either $f$ is strongly imprimitive or $f$ has a Zariski dense orbit, see Conjecture \ref{conj-zdo}.
We refer to \cite[Remark 1.13]{JXZ23} for a survey of known results.
The following theorem gives a different version of Theorem \ref{mainthm-A} from the aspect of ZDO.
In other words, for a non-isomorphic surjective endomorphism of a smooth projective threefold, we can reduce Conjecture ZDO to the case of a $\Q$-factorial terminal projective threefold with no birational contractions, but only (automatically $f$-equivariant) Fano contractions. 

\begin{theorem}\label{mainthm-D}
Let $f:X\to X$ be a non-isomorphic surjective endomorphism of a smooth projective threefold $X$.
Suppose $f$ is weakly primitive and has no Zariski dense orbit.
Then $X$ is uniruled and any birational MMP of $X$ is $f$-equivariant after iteration.
\end{theorem}

\par \vskip 1pc
{\bf Acknowledgement.}
The authors would like to thank Fei Hu, Jia Jia, Chen Jiang, Yohsuke Matsuzawa, Joseph Silverman, Junyi Xie and Guolei Zhong for their valuable discussions and comments.
The authors are supported respectively by Science and Technology Commission of Shanghai Municipality (No. 22DZ2229014) 
and a National Natural Science Fund; and ARF: A-8000020-00-00 of NUS. 
Many thanks to IMS at NUS, NJU at Nanjing, and ECNU at Shanghai for the warm hospitality and kind support.

\section{Preliminaries}

\begin{notation*}\label{notation}
Let $X$ and $Y$ be projective varieties of dimension $n$.
Let $f:X\to X$ be a surjective endomorphism and $\pi:X\to Y$ a finite surjective morphism.
Two $\R$-Cartier divisors $D_i$ of $X$ are {\it numerically equivalent}, denoted by $D_1\equiv D_2$, if $(D_1 - D_2)\cdot C=0$ for any curve $C$ on $X$.
Two $r$-cycles $C_i$ of $X$ are {\it weakly numerically equivalent}, denoted by $C_1 \equiv_w C_2$, if
$(C_1 - C_2)\cdot L_1 \cdots L_{n - r} = 0$ for
all Cartier divisors $L_i$.
For $\R$-Cartier divisors, the numerical equivalence coincides with the weak numerical equivalence; see \cite[Section 2]{MZ18}.

We use the following notation throughout the paper unless otherwise stated.
\renewcommand*{\arraystretch}{1.1}

\begin{longtable}{p{2cm} p{9cm}}
$\Pic(X)$    &the group of Cartier divisors of $X$ modulo linear equivalence~$\sim$\\
$\Pic_{\mathbb{K}}(X)$    &$\Pic(X)\otimes_{\Z}\mathbb{K}$ with $\mathbb{K}=\mathbb{Q}, \mathbb{R},\mathbb{C}$\\
$\Pic^0(X)$    &the neutral connected component of $\Pic(X)$\\
$\Pic^0_{\mathbb{K}}(X)$    &$\Pic^0(X)\otimes_{\Z}\mathbb{K}$ with $\mathbb{K}=\mathbb{Q}, \mathbb{R},\mathbb{C}$\\
$\NS(X)$    &$\Pic(X)/\Pic^0(X)$, the N\'eron-Severi group\\
$\N^1(X)$    &$\NS(X)\otimes_{\Z}\mathbb{R}$, the space of $\R$-Cartier divisors modulo numerical equivalence $\equiv$\\
$\NS_{\mathbb{K}}(X)$    &$\NS(X)\otimes_{\Z}\mathbb{K}$ with $\mathbb{K}=\mathbb{Q}, \mathbb{R},\mathbb{C}$\\
$\N_r(X)$    &the space of $r$-cycles modulo weak numerical equivalence~$\equiv_w$\\
$f^*|_V$    & the pullback action on $V$, which is any group or space above\\
$f_*|_V$    & the pushforward action on $V$, which is any group or space above\\
$\Nef(X)$   & the cone of nef classes in $\N^1(X)$\\
$\Amp(X)$   & the cone of ample classes in $\N^1(X)$\\
$\NE(X)$   & the cone of pseudo-effective classes in $\N_1(X)$\\
$\PE^1(X)$  & the cone of pseudo-effective classes in $\N^1(X)$\\
$\Supp D$   & the support of effective divisor $D=\sum a_i D_i$ which is $\bigcup_{a_i >0} D_i$, where the $D_i$ are prime divisors\\
$R_{\pi}$    & the ramification divisor of $\pi$ assuming that $X$ and $Y$ are normal\\
$B_{\pi}$    & $\pi(\Supp R_f)$ the (reduced) branch divisor of $\pi$ assuming that $X$ and $Y$ are normal\\
$\kappa(X,D)$ & Iitaka dimension of a $\Q$-Cartier divisor $D$\\
$\rho(X)$    & Picard number of $X$ which is $\dim_{\R} \N^1(X)$
\end{longtable}

\end{notation*}

\begin{convention*}
We refer to \cite{KM98,BCHM10} for the standard notation and facts for singularities ({\it terminal, klt, lc}) and divisors involved in the minimal model program (MMP for short).
Throughout this article, our minimal model program of $X$ involves no boundary part.
By a divisorial contraction, a flipping contraction, a flip, or a Fano contraction, we always mean an operation of some $K_X$-negative extremal ray.
\end{convention*}

We recall the following type of Hodge index theorem which is known to experts, see also \cite[Lemma 4.1]{MZg22}.
\begin{lemma}(cf.\,\cite[Corollarie 3.2]{DS04} or \cite[Lemma 3.2]{Zha16})
\label{lem-hodge}
Let $X$ be a normal projective variety. 
Let $D_1\not\equiv0$ and $D_2\not\equiv 0$ be two nef $\mathbb{R}$-Cartier divisors such that $D_1\cdot D_2\equiv_w 0$. 
Then $D_1\equiv aD_2$ for some $a>0$.
\end{lemma}

The following type of projection formula (cf.~\cite[Proposition 3.7]{MZ22}) will be frequently used throughout this paper.
\begin{proposition}\label{prop-f*f*}
Let $f:X\to X$ be a surjective endomorphism of a projective variety $X$.
Then $f_*f^*=f^*f_*=(\deg f)\id$ on $\Pic_{\mathbb{K}}(X)$.
In particular, $f^*P\sim_{\Q} aP$ holds for some rational number $a>0$ when $P$ is an $f$-invariant $\Q$-Cartier prime divisor;
and Iitaka dimension satisfies $\kappa(X, Q) = \kappa(X, f^*Q) = \kappa(X, f_*Q)$ if $Q$ and $f_*Q$ are $\Q$-Cartier divisors
(see \cite[Theorem 5.13]{Uen75}, or Lemma \ref{lem-pushforward-fiitka} below).
\end{proposition} 

\begin{definition}
Let $f:X\to X$ be a surjective endomorphism of a variety~$X$.
A subset $Z\subseteq X$ is said to be {\it $f$-invariant} (resp. {\it $f^{-1}$-invariant}) if $f(Z)=Z$ (resp.~$f^{-1}(Z)=Z$), and $Z$ is said to be {\it $f$-periodic} (resp. {\it $f^{-1}$-periodic}) if $f^s(Z)=Z$ (resp.~$f^{-s}(Z)=Z$) for some $s>0$.
\end{definition}

\begin{definition}
Let $f:X\to X$ be a surjective endomorphism of a projective variety $X$. 
A prime divisor $E$ is said {\it $f$-prime} if $f^i(f^j(E))$ is irreducible for any $i, j\in \Z$.
Denote by 
$$E_n:=f^n(E)$$
where $n\in \mathbb{Z}$.
Note that an $f$-prime divisor $E$ is $f$-periodic if and only if it is $f^{-1}$-periodic.
\end{definition}

\begin{lemma}\label{lem-fprime}
Let $f:X\to X$ be a surjective endomorphism of a $\Q$-factorial normal projective variety $X$. 
Let $E$ be a prime divisor with $\kappa(X,E)=0$.
Then $E$ is $f$-prime.
\end{lemma}
\begin{proof}
By Proposition \ref{prop-f*f*}, we have
$$f^*f_*=f_*f^*=(\deg f) \id$$ on $\Pic_{\Q}(X)$.
Suppose $f(E)=f(E')$ for some other prime divisor $E'$.
Then $E\sim_{\Q} aE'$ for some rational number $a>0$ and hence $\kappa(X,E)>0$, a contradiction.
Let $E_i:=f^i(E)$.
We have $f^{-1}(E_{i+1})=E_i$ and 
$$\kappa(X,E_{i+1})=\kappa(X, f^*E_{i+1})=\kappa(X, E_i)=0$$ when $i>0$ by \cite[Theorem 5.13]{Uen75}.
Similarly, $E_i$ is irreducible when $i<0$.
\end{proof}

\begin{definition}
Let $X$ be a projective variety.
The {\it numerical Cartier index} of $X$ is the minimal positive integer $r$ such that $rD\equiv  D'$ with $D'$ Cartier for any $\Q$-Cartier Weil divisor $D$.
The numerical Cartier index is always finite by the following lemma.
\end{definition}

The following lemma is suggested by Chen Jiang by using a similar proof of \cite[Lemma 4.1]{MZ22}.

\begin{lemma}\label{lem-index}
Let $X$ be a projective variety.
Then there is a positive integer $r$ such that $rD\equiv D'$ with $D'$ Cartier for any $\Q$-Cartier Weil divisor $D$.
\end{lemma}
\begin{proof}
Let $n:=\dim X$.
Choose a basis $\{[H_i]\}_{i=1}^m$ of $\N^1(X)$ where $H_i$ is ample and Cartier.
Consider the following space
$$\N^{n-1}_{\mathbb{K}}(X):=\{\sum_{\text{\rm finite sum}}
a x_1\cdots x_{n-1}\,|\, a\in\mathbb{K}, x_1,\cdots, x_{n-1} \text{ are Cartier divisors}\}/\equiv_w,$$
see also \cite[Section 3]{Men20}.
Note that $\N^{n-1}_{\mathbb{R}}(X)$ can be identified with $\N_1(X)$ by \cite[Lemma 3.2]{Zha16}. 
Choose a basis $\{[C_j]\}_{j=1}^m$ of $\N^{n-1}_{\mathbb{Z}}(X)$.
Then the intersection of $C_j$ with any Weil divisor is an integer.
Let $A$ be the intersection matrix $(H_i\cdot C_j)_{i,j}$ which is invertible and integral.
Let $D$ be a $\Q$-Cartier Weil divisor.
Then there exists $a_j\in \Q$ such that 
$$D\cdot C_j=\sum_{i=1}^m a_i H_i\cdot C_j\in \mathbb{Z}$$ holds for each $i$
and hence
$$(a_1,\cdots, a_m)\in A^{-1}(\mathbb{Z}^m)\subseteq \mathbb{Z}^m/\det(A).$$
We are done by letting $r =\det(A)$. 
\end{proof}

\begin{definition}\label{def-awd}
Let $\pi:X\dashrightarrow Y$ be a dominant rational map of projective varieties.
$\pi$ is {\it almost well-defined} (or {\it almost holomorphic} when $K = \C$) if it is well defined outside a Zariski closed proper subset of $Y$.
\end{definition}

\begin{definition}\label{def-hv}
Let $\pi:X\dashrightarrow Y$ be a dominant rational map of normal projective varieties.
Let $D=\sum\limits_{i=1}^n a_iD_i$ be a Weil $\R$-divisor of $X$ where $D_i$ are prime divisors.
Assume that $D_i$ dominates $Y$ when $i\le r$ and $D_i$ does not dominate $Y$ when $i>r$.
Denote by $D^h:=\sum\limits_{i=1}^r a_i D_i$ the {\it $\pi$-horizontal part} of $D$ and $D^v:=\sum\limits_{i=r+1}^n a_i D_i$ the {\it $\pi$-verticle part} of $D$.
\end{definition}

\begin{definition}\label{def-d1}(Dynamical degree; $\delta_f$, $\iota_f$)
Let $f:X\to X$ be a surjective endomorphism of a projective variety $X$.
The (first) {\it dynamical degree} of $f$ is defined by
$$\delta_f=\lim\limits_{m\to+\infty}((f^m)^*H\cdot H^{\dim X-1})^{1/m},$$
where $H$ is any nef and big Cartier divisor of $X$.
This limit exists and is independent of the choice of $H$.
Since our $f$ is well-defined, $\delta_f$ is the spectral radius of $f^*|_{\N^1(X)}$.
Denote by $\iota_f$ the {\it minimal modulus of eigenvalues} of $f^*|_{\N^1(X)}$.

Consider the equivariant dynamical systems of projective varieties
$$\xymatrix{
f \acts X \ar@{-->}[r]^{\pi} &Y\racts g
}$$
where $\pi$ is a dominant rational map.
Define the \textit{relative (first) dynamical degree} of $f$ as
$$\delta_{f|_{\pi}}= \lim_{n \to \infty} ((f^n)^*H_X \cdot (\pi^*H_Y)^{\dim Y}\cdot H_X^{\dim X-\dim Y -1})^{1/n}$$
where $H_X$ and $H_Y$ are respectively nef and big Cartier divisors of $X$ and $Y$.
The definition is independent of choices of $H_X$ and $H_Y$.
Note that $\delta_{f|_\pi}$ is just the degree of $f$ on the general fibre of $\pi$ when $\dim X=\dim Y+1$.

We recall the following facts.
\begin{enumerate}
\item $\delta_f=\deg f$ if $\dim X=1$.

\item $\delta_{f|_Z}\le \delta_f$ for any $f$-invariant irreducible closed subvariety $Z\subseteq X$.

\item Product formula: $\delta_f=\max\{\delta_{f|_{\pi}}, \delta_g\}$.
In particular, $\delta_f=\delta_g$ if $\pi$ is generically finite and dominant.

\item $\delta_{f_1\times f_2}=\max\{\delta_{f_1},\delta_{f_2}\}$.

\item $\delta_{f^n}=\delta_f^n$ for any $n>0$.
\end{enumerate}
We refer to \cite[Notation 2.1]{MMS+22} for a more general setting and references.
\end{definition}

\begin{lemma}\label{lem-imprimitive}
Consider the equivariant dynamical systems of projective varieties:
$$\xymatrix{
\save[]+<3.1pc,0.05pc>*{f \acts} \restore &X\ar@{-->}[ld]_\phi\ar@{-->}[rd]^\psi\\
g \acts Y&&Z \racts h
}$$
where $\phi$ and $\psi$ are dominant rational maps.
Suppose the induced map $\sigma:X\dashrightarrow Y\times Z$ is generically finite.
Then either $\delta_f=\delta_g$ or $\delta_f=\delta_h$.
\end{lemma}
\begin{proof}
Note that $\delta_{g\times h}=\max\{\delta_g, \delta_h\}\le \delta_f$ by the dynamical product formula.
Let $\overline{X}$ be the closure of the image $\sigma(X)$ in $Y\times Z$ and $\overline{f}:=(g\times h)|_{\overline{X}}$.
Then $\delta_{\overline{f}}\le \delta_{g\times h}$.
Since $\sigma$ is generically finite, we have $\delta_f=\delta_{\overline{f}}\le \max\{\delta_g, \delta_h\}\le \delta_f$.
\end{proof}

\begin{lemma}\label{lem-d1-weil}
Let $f:X\to X$ be a surjective endomorphism of a projective variety $X$ of dimension $n$.
Suppose $f^*D\equiv aD$ for some $D\in \N_{n-1}(X)\backslash \{0\}$.
Then $|a| \le \delta_f$.
\end{lemma}
\begin{proof}
Let $b$ be the spectral radius of $f^*|_{\N_{n-1}(X)}$,
so $b\ge |a|$.
Let $V$ be the cone of pseudo-effective cycles in $\N_{n-1}(X)$.
Applying the Perron-Frobenius theorem to $f^*|_V$, $f^*D'\equiv bD'$ for some $D'\in V\backslash\{0\}$.
Let $H$ be an ample Cartier divisor such that and $H-D'\in V$.
Note that $D'\cdot H^{n-1}>0$.
Then the lemma follows from the upper bound of $b$ below:
$$\delta_f=\lim\limits_{m\to+\infty}((f^m)^*H\cdot H^{n-1})^{1/m}\ge \lim\limits_{m\to+\infty}((f^m)^*D'\cdot H^{n-1})^{1/m}=b.$$
\end{proof}

\begin{definition}\label{def-a-deg}(Weil height function and arithmetic degree)
Let $X$ be a normal projective variety defined over $\overline{\mathbb{Q}}$.
We refer to \cite[Part B]{HS00}, \cite{KS16b} or \cite[\S 2.2]{Mat20a} for the detailed definition of the {\it Weil height function} 
$h_D:X(\overline{\mathbb{Q}})\to \R$ associated with some $\R$-Cartier divisor $D$ on $X$.
Here, we simply list some fundamental properties of the height function.
\begin{itemize}
\item $h_D$ is determined by the divisor $D$ only up to a bounded function.
\item $h_{\sum a_i D_i}=\sum a_i h_{D_i}+O(1)$ where $O(1)$ means some bounded function.
\item $h_E$ is bounded below outside $\Supp E$ for any effective Cartier divisor $E$.
\item Let $\pi:X\to Y$ be a surjective morphism of normal projective varieties and $B$ some $\R$-Cartier divisor of $Y$.
Then $h_B(\pi(x))=h_{\pi^*B}(x)+O(1)$ for any $x\in X(\overline{\mathbb{Q}})$.
\end{itemize}

The {\it arithmetic degree} $\alpha_f(x)$ of $f$ at $x \in X(\overline{\mathbb{Q}})$ is defined as
$$\alpha_f(x) = \lim\limits_{m\to+\infty} \max\{1, h_H (f^m(x))\}^{1/m},$$
where $H$ is an ample Cartier divisor.
This limit exists and is independent of the choice of $H$ (cf.~\cite[Theorem 2]{KS16a}, \cite[Proposition 12]{KS16b}).
Moreover, $\alpha_f(x)$ is either $1$ or the absolute value of an eigenvalue of $f^*|_{\N^1(X)}$ (cf.~\cite[Remark 23]{KS16a}).
Note that $\alpha_f(x)\le \delta_f$ and $\alpha_{f^s}(x)=\alpha_f(x)^s$.
This allows us to replace $f$ by any positive power whenever needed.
\end{definition}

We recall the following lemma for the induction purpose of KSC, see \cite[Lemma 5.6]{Mat20a} and \cite[Lemma 2.5]{MZ22}.
\begin{lemma}\label{lem-ksc-iff}
Consider the equivariant dynamical systems of projective varieties
$$\xymatrix{
f \acts X \ar@{-->}[r]^{\pi} &Y\racts g
}$$
where $\pi$ is a dominant rational map.
Then the following hold.
\begin{itemize}
\item[(1)] Suppose $\pi$ is generically finite.
Then KSC holds for $f$ if and only if it holds for $g$.
\item[(2)] Suppose $\delta_f=\delta_g$ and KSC holds for $g$.
Then KSC holds for $f$.
\end{itemize}
\end{lemma}

\section{Dynamics on surfaces}

The following Lemma \ref{lem-surf-toric} is used in the proof of Lemma \ref{lem-fano-toric}.

\begin{definition}
Let $f:X\to Y$ be a non-isomorphic finite surjective morphism of projective varieties.
We say $f$ is {\it totally ramified} if $f^{-1}Q$ is prime for any primed divisor $Q\subseteq B_f$.
Let $y\in Y$ be a closed point.
We say $y$ is an {\it $f$-unsplitting point} if $f^{-1}(y)$ is a single point. 
\end{definition}

\begin{lemma}\label{lem-surf-toric}
Let $f:X\to Y$ be a totally ramified, non-isomorphic, finite surjective morphism of del Pezzo surfaces with $\rho(X)=\rho(Y)$.
Suppose $(Y, B_f)$ is log Calabi-Yau.
Then $(X,\Supp R_f){\stackrel{\mathclap{\sigma}}{\cong}} (Y, B_f)$ are toric pairs and $\Sing(B_f)$ is the set of $f$-unsplitting points.
\end{lemma}

\begin{proof}
There are exactly $5$ isomorphic classes of toric del Pezzo surfaces 
$$\mathbb{P}^1\times \mathbb{P}^1, \mathbb{P}^2, S_j \, (j \le 3) $$ where 
$S_j$ is the blowup of $\mathbb{P}^2$ at $j$ distinct points (not on the same line when $j = 3$).
We remark (for late use) that there is no finite surjective morphism from $\mathbb{P}^1\times \mathbb{P}^1$ to $S_1$, or the other way round, 
since the latter has a negative curve while the former has none.

Suppose we have proved that
$(X,\Supp R_f)$ and $(Y, B_f)$ are toric pairs. Then there is an isomorphism $\sigma:X\to Y$ such that $\sigma^{-1}(B_f)=\Supp R_f$ since $\rho(X)=\rho(Y)$.
Under the identification of $\sigma$, we can regard $f$ as a surjective endomorphism such that $f^{-1}(B_f)=B_f$ since $f$ is totally ramified.
Since $(Y,B_f)$ is log Calabi-Yau and $Y$ is Fano, $B_f$ is ample and $R_f$ is big.
Now $f^*B_f-B_f=R_f$ implies that $f$ is int-amplified (cf.~\cite[Theorem 1.1]{Men20}).
Let $Q$ be an irreducible component $Q$ of $B_f$.
After iteration, we may assume $f^{-1}(Q)=Q$.
So $f|_Q$ is int-amplified for any irreducible component $Q$ of $B_f$.
Note that $B_f$ is a simple normal crossing loop of smooth rational curves.
Then $B_{f|_Q}=(B_f-Q)\cap Q$ and hence the set of $f$-unsplitting points is exactly $\Sing(B_f)=\Sing(\Supp R_f)$.
Thus it suffices to show $(X,\Supp R_f)$ and $(Y, B_f)$ are toric pairs. 

We first consider two relatively minimal cases of $Y$.

\textbf{Case $Y\cong \mathbb{P}^2$.}
Note that $X\cong \mathbb{P}^2$ since $\rho(X)=\rho(Y)=1$.
Since $f$ is totally ramified, $B_f$ is a union of lines by \cite[Theorem]{Gur03}.
Since $(Y, B_f)$ is log Calabi-Yau and hence LC, $B_f$ is a union of three lines that do not intersect at one point.
Therefore, $(Y, B_f)$ is a toric pair.
Since $f$ is totally ramified, $\Supp R_f$ has three irreducible components.
Note that $K_X+\Supp R_f=f^*(K_Y+B_f)$.
So $(X, \Supp R_f)$ is log Calabi-Yau (cf.~\cite[Proposition 5.20]{KM98}). Hence clearly it is a toric pair.

\textbf{Case $Y\cong \mathbb{P}^1\times \mathbb{P}^1$.}
Since $\rho(X)=\rho(Y)=2$, $X$ is either $\mathbb{P}^1\times \mathbb{P}^1$ or $S_1$.
The existence of $f: X \to Y$ implies that $X = \mathbb{P}^1\times \mathbb{P}^1$ as remarked above.
Then 
$f$ splits as two surjective endomorphisms $g, h$ of $\mathbb{P}^1$.
Since $(Y, B_f)$ is log Calabi-Yau and hence LC,
$B_f= \sum_{j=1}^2 (\mathbb{P}^1\times \{a_j\}+\{b_j\}\times \mathbb{P}^1)$ with $a_1\neq a_2$ and $b_1\neq b_2$.
Since $f$ is totally ramified, $g$ and $h$ are totally ramified.
Therefore, $(Y, B_f)$ and $(X, \Supp R_f)$ are toric pairs.

Next, we consider the case when $Y$ is not relatively minimal and show by induction on $\rho(Y)$.
Let $C_X$ be a $(-1)$-curve on $X$ and $C_Y:=f(C_X)$.
Note that $f^*C_Y=qC_X$ for some integer $q>0$;
so $C_Y^2<0$ and hence $C_Y$ is a $(-1)$-curve on $Y$ since $K_Y\cdot C_Y<0$
(cf. Proposition \ref{prop-f*f*}).
Note that $\deg f>1$ since $B_f$ is non-empty.
So $q^2=\deg f>1$.
Then $\deg f|_{C_X}=q>1$ and $C_Y\subseteq B_f$.
Consider the following commutative diagram
$$\xymatrix{
X\ar[r]^f\ar[d]_{\pi_{C_X}} &Y\ar[d]^{\pi_{C_Y}}\\
X'\ar[r]^{f'} &Y'
}
$$
where $\pi_{C_X}$ and $\pi_{C_Y}$ are the blowdowns of $C_X$ and $C_Y$ respectively.
By the rigidity lemma (cf.~\cite[Lemma 1.15]{Deb01}), $f'$ is well-defined. 
Note that $\rho(X')=\rho(Y')=\rho(X)-1$ and $f'$ is finite surjective. 
Moreover, $f'$ is totally ramified and $B_{f'}=\pi_{C_Y}(B_f)$.
Then $K_{Y'}+B_{f'}={\pi_{C_Y}}_*(K_Y+B_f)\equiv 0$ and hence $K_Y+B_f=\pi_{C_Y}^*(K_{Y'}+B_{f'})$.
So $(Y', B_{f'})$ is log Calabi-Yau and hence a toric pair by induction.
Note that $\Sing(B_{f'})$ is the set of $f'$-unsplitting points which contains $\pi_{C_Y}(C_Y)$.
Then $B_{f'}$ has multiplicity $2$ at $\pi_{C_Y}(C_Y)$.
So $\pi_{C_Y}(C_Y)$ is a closed toric orbit and $\pi_{C_Y}$ is a toric blowup.
In particular, $(Y, B_f)$ is a toric pair.
Similarly, $(X, \Supp R_f)$ is a toric pair.
\end{proof}

For convenience, we give a proof of the lemma below, which should be well-known. 

\begin{lemma}\label{lem-surf-deg}
Let $f:X\to X$ be a surjective endomorphism of a projective surface $X$.
Then $\deg f=\delta_f\cdot \iota_f$.
\end{lemma}
\begin{proof} 
Consider the Jordan canonical form of $f^*|_{\NS_{\C}(X)}$.
Let $x_1,\cdots, x_r$ be a basis of $\NS_{\C}(X)$ such that either $f^*x_i=\lambda_ix_i$ or $f^*x_i=\lambda_ix_i+x_{i+1}$.
Note that $x_1\cdot x_j\neq 0$ and $x_1\cdot x_{j+1}=0$ for some $j$ (write $x_{r+1}=0$).
By the projection formula, $\deg f=\lambda_1\cdot \lambda_j$.
If we take $\lambda_1=\delta_f$,
then $\deg f\ge \delta_f\cdot \iota_f$.
If we take $\lambda_1=\iota_f$,
then $\deg f\le \iota_f\cdot \delta_f$.
\end{proof}

We shall use the following lemma in the proof of Theorem \ref{thm-fiitaka2}.
\begin{lemma}\label{lem-surf-deg=1}
Let $f:X\to X$ be a non-isomorphic surjective endomorphism of a normal projective surface $X$.
Let $C$ be an $f$-prime curve which is not $f^{-1}$-periodic.
Then $X$ is $\Q$-factorial and $\kappa(X,C)>0$.
\end{lemma}
\begin{proof}
If $f$ is int-amplified, then any $f$-prime curve $C$ is $f^{-1}$-periodic by \cite[Lemma 8.1]{Men20}.
So we may assume $f$ is not int-amplified.
By \cite[Theorem 1.1]{JXZ23}, $X$ is further $\Q$-factorial.
Indeed, if $K_X$ is pseudo-effective, then $X$ has quotient singularities and hence is $\Q$-factorial.
If $K_X$ is not pseudo-effective, then $\rho(X)=2$ and $X$ admits a Fano contraction to a curve by \cite[Theorem 5.4]{MZ22}.
By \cite[Proposition 2.33]{Nak17}, $X$ is $\Q$-factorial.

Let $C_i:=f^i(C)$ with $i\in \mathbb{Z}$.
Since $C$ is not $f^{-1}$-periodic, $f^*C_{i+1}=C_i$ when $i\gg 1$.
By the projection formula, 
$$C_i\cdot C_j=f^*C_{i+1}\cdot f^*C_{j+1}=(\deg f) C_{i+1}\cdot C_{j+1}.$$
Note that $\deg f>1$.
Then we have
$$\lim_{i,j\to+\infty} C_i\cdot C_j=0.$$
By Lemma \ref{lem-index}, we have
$$r^2C_i\cdot C_j\in\mathbb{Z}$$ 
where $r$ is the numerical Cartier index of $X$.
So we have $C_i\cdot C_j=0$ for $i,j\gg 1$.
By the projection formula again, this holds for any $i,j\in \mathbb{Z}$.
Thus, $C_i\cap C_j=\emptyset$ for any $i\neq j$.
By \cite[Theorem 1.1]{BPS16}, each $C_i$ is semi-ample and hence $\kappa(X,C)>0$.
\end{proof}

\begin{theorem}\label{thm-dbbc-surf}
Let $f:X\to X$ be a non-isomorphic surjective endomorphism of a normal projective surface $X$.
Then $f$ is either $\delta$-imprimitive, quasi-abelian, or polarized after iteration.
In particular, Conjecture \ref{mainconj} holds for surfaces.
\end{theorem}
\begin{proof}
By \cite[Theorems 5.1 and 5.4]{MZ22}, we only need to consider the case where
$f$ is quasi-\'etale, and $X$ is a quasi-\'etale quotient of $E\times T$ with $E$ an elliptic curve and $T$ a curve of genus $\ge 2$.
Then $\kappa(X,K_X)=1$ and $K_X$ is semi-ample.
By \cite[Theorem A]{NZ09}, after iteration, we have the following equivariant dynamical systems:
$$\xymatrix{
f \acts X \ar[r]^{\varphi} &Y\racts g=\id
}$$
where $\varphi$ is the Iitaka fibration and $\dim Y=1$.
Note that $\delta_f=\deg f|_{X_y}=\deg f>1$ and $X$ is non-uniruled.
So the general fibre $X_y$ of $\varphi$ is an elliptic curve.
Note that $f|_{X_y}$ is not a translation.
So the set of fixed points $\text{Fix}(f|_{X_y})$ is non-empty for general $y\in Y$.
In particular, $\text{Fix}(f)$ contains an irreducible curve $C$ such that $\varphi|_{C}:C\to Y$ is finite surjective.
Note that $f|_C=\id$ and $f^*C\sim_{\Q} (\deg f)C$ by Proposition \ref{prop-f*f*}, so $C^2 = 0$.
If $f^{-1}(C)=C$, then $f^*C=(\deg f)C$ and hence $C\subseteq B_f$, a contradiction since $f$ is quasi-\'etale.
So there exists some irreducible curve $C'\neq C$ such that $f(C')=C$.
Now $f^*C\sim_{\Q} (\deg f)C$ and Proposition \ref{prop-f*f*} imply
$f^*C'\sim_{\Q} (\deg f) C'$.
Since $f|_{C'} = \id$ too and 
$$(f|_C)^*(C'|_C)\sim_{\Q} (\deg f)C'|_C ,$$
$C'\cap C=\emptyset$. Hence $C$ is semi-ample, so 
$\kappa(X,C)=1$ since $C^2 = 0$.
Let $\psi:X\to Z$ be the Iitaka fibration of $C$.
Then $\psi$ is $f$-equivariant and $\delta_{f|_Z}=\delta_f$ by the dynamical product formula.
Thus $f$ is $\delta$-imprimitive.
\end{proof}

\begin{remark}
By the proof of Theorem \ref{thm-dbbc-surf}, we have indeed proved that ``strongly $f$-imprimitive'' is ``$\delta$-imprimitive'' for surfaces.
\end{remark}

\section{Dynamical Iitaka fibration}

In this section, we introduce the new concepts: {\it dynamical Iitaka dimension} and {\it dynamical Iitaka fibration}.

\begin{definition}[Dynamical Iitaka dimension]
Let $X$ be a normal projective variety and $D$ a $\Q$-Cartier divisor.
Denote by $V_f(D)$ the subspace of $\Pic_{\Q}(X)$ spanned by $D_i:=(f^*)^i(D)$ with $i\in\Z$.
By \cite[Proposition 3.7]{MZ22} and noting that $f^*|_{\Pic_{\Q}(X)}$ is invertible, $V_f(D)$ is finite dimensional.
We define the {\it dynamical $f$-Iitaka dimension} of $D$ as 
$$\kappa_f(X,D):=\max\{\kappa(X,D')\,|\,D'\in V_f(D)\}.$$
If $D_0,\cdots, D_n$ are effective and span $V_f(D)$, 
then $\kappa_f(X,D)=\kappa(X, \sum\limits_{i=0}^n D_i)$.
\end{definition}

\begin{lemma}\label{lem-fkappa>0}
Let $f:X\to X$ be a surjective endomorphism of a normal projective variety $X$.
Let $D$ be a $\Q$-Cartier effective divisor of $X$ such that $\Supp D$ is not $f^{-1}$-periodic.
Then $\kappa_f(X,D)>0$.
\end{lemma}
\begin{proof}
Let $D_i:=(f^*)^i(D)$.
By \cite[Lemma 19]{KS16a}, there is a monic integral polynomial 
$$P_f(t)=\sum_{i=0}^n a_i t^i\in \mathbb{Z}[t]$$ such that $P_f(f^*)$ annihilates $\Pic_{\Q}(X)$.
So we have $$A:=\sum_{a_i>0} a_i D_i\sim_{\Q} \sum_{a_i<0} -a_i D_i=:B$$
where $\Supp D_n\subseteq \Supp A$.
Since $\Supp D$ is not $f^{-1}$-periodic, we have $\Supp D_n\not\subseteq \bigcup\limits_{i=0}^{n-1} \Supp D_i$.
In particular, $A\neq B$ and hence $\kappa_f(X,D)\ge \kappa(X,A)>0$.
\end{proof}

\begin{lemma}\label{lem-pushforward-fiitka}
Consider the equivariant dynamical systems of normal projective varieties
$$\xymatrix{
f \acts X \ar@{-->}[r]^{\pi} &Y\racts g
}$$
Then the following hold.
\begin{enumerate}
\item If $\pi$ is a surjective morphism, then $\kappa_f(X,\pi^*D)= \kappa_g(Y,D)$ for any $\Q$-Cartier divisor $D$ of $Y$.
\item If $\pi$ is a birational map and $Y$ is $\Q$-factorial, then $\kappa_f(X,D)\le  \kappa_g(Y,\pi_*D)$ for any $\Q$-Cartier divisor $D$ of $X$.
\end{enumerate}
\end{lemma}
\begin{proof}
(1) Note that $V_f(\pi^*D)=\pi^*(V_g(D))$ and $\kappa(X,\pi^*D)= \kappa(Y, D)$ for any $\Q$-Cartier divisor $D$ of $Y$ (cf.~\cite[Theorem 5.13]{Uen75}).
Then $\kappa_f(X,\pi^*D)=\kappa_g(Y,D)$.

(2) Let $W$ be the normalization of the graph of $\pi$ and $h:W\to W$ the lifting.
Let $p_X:W\to X$ and $p_Y:W\to Y$ be the two projections.
Then $\kappa_f(X,D)=\kappa_h(W,p_X^*D)$ by (1).
Note that $\pi_*D={p_Y}_*p_X^*D$.
So we may assume that $\pi$ is well-defined.
We may also assume that $D$ is effective such that $\kappa_f(X,D)=\kappa(X,D)$.
Then $n\pi^*\pi_*D=D+E$ with $E$ effective when $n\gg 1$.
Then $\kappa_g(Y,\pi_*D)\ge \kappa(X,\pi^*\pi_*D)\ge \kappa(X,D)=\kappa_f(X,D)$.
\end{proof}

We recall the Chow reduction as introduced in \cite[Proposition 4.14, Definition 4.15 and Theorem 4.19]{Nak10}; see also \cite[Theorem 7.2]{MZ22} and the statement before it.

\begin{theorem}\label{thm-chow}
Let $\pi:X\dashrightarrow Y$ be a dominant rational map of normal projective varieties defined over an algebraically closed field. 
Suppose $\pi$ has connected fibres.
Then there exist a normal projective variety $Z$ and a birational map $\mu:Y\dashrightarrow Z$ satisfying the following conditions:
\begin{enumerate}
\item Let $\Gamma_Z\subseteq X\times Z$ be the graph of the composite $\mu\circ \pi: X\dashrightarrow Y\dashrightarrow Z$. 
Then $\Gamma_Z\to Z$ is equi-dimensional.
\item Let $\mu':Y\dashrightarrow Z'$ be a birational map to another normal projective variety $Z'$ such that the induced projection $\Gamma_{Z'}\to Z'$ is equi-dimensional.
Then there exists a birational morphism $\nu:Z'\to Z$ with $\mu=\nu\circ\mu'$.
\end{enumerate}
In particular, for any surjective endomorphism $f:X\to X$ and a dominant self-map $g:Z\dashrightarrow Z$ with $g\circ (\mu\circ \pi)=(\mu\circ \pi)\circ f$, the map $g$ is 
a surjective endomorphism.
\end{theorem}

\begin{definition}
We call the composition $\mu\circ \pi:X\dashrightarrow Z$ in Theorem \ref{thm-chow} the {\it Chow reduction} of $\pi$ which is unique by the condition (2).
\end{definition}

\begin{theorem}\label{thm-fibration}
Let $f:X\to X$ be a surjective endomorphism of a normal projective variety $X$.
Let $D$ be a $\Q$-Cartier divisor with $\kappa_f(X,D)\ge 0$.
Then there is an $f$-equivariant dominant rational map $\phi_{f,D}:X\dashrightarrow Y$ to a normal projective variety $Y$
(with $f|_Y$ a surjective endomorphism too)
satisfying the following conditions.
\begin{enumerate}
\item $\dim Y=\kappa_f(X,D)$.
\item Let $\Gamma$ be the graph of $\phi_{f,D}$. Then the induced projection $\Gamma\to Y$ is equi-dimensional.
\item $\phi_{f,D}$ is birational to the Iitaka fibration of any $D'\in V_f(D)$ with $\kappa(X,D')=\kappa_f(X,D)$. 
\item Suppose further $D$ is an $f$-prime divisor dominating $Y$. 
Then $D$ is $f^{-1}$-periodic.
\end{enumerate}
\end{theorem}

\begin{proof}
Let $D_i\in V_f(D)$ with $\kappa(X,D_i)=\kappa_f(X,D)$ and $i=1,2$.
Let $\phi_{D_i}: X\dashrightarrow Y_{D_i}$ be the Iitaka fibration of $D_i$.
Note that 
$$\kappa_f(X, D)\ge \kappa(X, D_1+D_2)\ge \kappa(X, D_1)=\kappa_f(X,D).$$
So we have
$$\kappa(X, D_1+D_2)=\kappa(X, D_1)=\kappa(X, D_2).$$
Therefore, $\phi_{D_i}=\sigma_i\circ \phi_{D_1+D_2}$ for some birational maps $\sigma_i:Y_{D_1+D_2}\dashrightarrow Y_{D_i}$ (cf.~\cite[Lemma 7.3]{MZ22}).
Then we simply take $\phi_{f,D}$ as the Chow reduction of $\phi_{D_1}$ in Theorem \ref{thm-chow}.
This proves (1), (2), and (3).

For (4), suppose the contrary that $D$ is $f$-prime and dominating $Y$ and yet it is not $f^{-1}$-periodic.
Let $P_i:=(f^*)^i(D)$.
Then $\Supp P_i$ is irreducible and $\Supp P_i\neq \Supp P_j$.
Let $n\ge 0$ be the minimal integer such that $P_0,\cdots, P_n$ span $V_f(D)$.
Then we have
$$D':=P_{n+1}+\sum_{i\in I} a_i P_i\sim_{\Q} \sum_{j\in J} b_j P_j$$
where $a_i,b_j$ are positive rational numbers, $J\neq\emptyset$, and $I \cap J=\emptyset$.
Then $D'$ is $\Q$-movable.
Note that $\phi_{f,D'}:X\dashrightarrow Y_{D'}$ factors through $\phi_{f,D}$ and $D'$ does not dominate $Y_{D'}$.
So we get a contradiction.
\end{proof}

\begin{definition}\label{def-fiitaka}
We call $\phi_{f,D}$ in Theorem \ref{thm-fibration} the {\it dynamical $f$-Iitaka fibration} ($f$-Iitaka fibration for short) of $D$. It is unique by the uniqueness of the Chow reduction and Theorem \ref{thm-fibration} (3).
\end{definition}

A special case of the $f$-Iitaka fibration has been studied in \cite[Theorem 7.8]{MZ22}.
\begin{theorem}\label{thm-fiitaka-polarized}
Let $f:X\to X$ be a surjective endomorphism of a projective variety $X$.
Let $D$ be a $\Q$-Cartier divisor such that $f^*D\equiv qD$ for some $q>1$ and $\kappa(X,D)>0$.
Let $\phi:X\dashrightarrow Y$ be the $f$-Iitaka fibration of $D$.
Then $f|_Y$ is $q$-polarized.
\end{theorem}
\begin{proof}
The same proof of \cite[Theorem 7.8]{MZ22} works after replacing $-K_X$ by $D$.
\end{proof}

\section{Dynamical Iitaka fibration of an $f$-prime divisor}
In this section, we further characterize the dynamical Iitaka fibration of the exceptional divisor. 
Theorem \ref{thm-fiitaka2} is the main result.

Our first result of the section is on $f$-prime divisors.

\begin{theorem}\label{thm-fiitaka1}
Let $f:X\to X$ be a non-isomorphic surjective endomorphism of a $\Q$-factorial normal projective threefold $X$.
Let $E$ be an $f$-prime divisor of $X$. 
Then either $E$ is $f^{-1}$-periodic or $0<\kappa_f(X,E)<3$.
\end{theorem}

We do some preparation for the proof of Theorem \ref{thm-fiitaka1}.
Let $E_i:=f^i(E)$ with $i\in\mathbb{Z}$.
We may assume $E$ not $f^{-1}$-periodic.
Then $E_i\neq E_j$ for any $i\neq j$.
Moreover, $E_n\not\subseteq \Supp R_f$ when $|n|\gg 1$.
Here $R_f$ is the ramification divisor of $f$.
So we have 
$$f^*E_{n+1}=E_n$$ 
when $|n|\gg 1$.
Now the story of the non-isomorphic $f$ starts with the following easy but useful lemma.

\begin{lemma}\label{lem-ijk}
$E_i\cdot E_j\cdot E_k=0$ for any $i,j,k\in \Z$. 
\end{lemma}

\begin{proof}
Take $i,j,k\gg 1$, we have
$$E_i\cdot E_j\cdot E_k=f^*E_{i+1}\cdot f^*E_{j+1}\cdot f^*E_{k+1}=(\deg f) E_{i+1}\cdot E_{j+1}\cdot E_{k+1}$$
by the projection formula.
Note that $\deg f>1$. Then we have
$$\lim_{i,j,k\to+\infty}E_i\cdot E_j\cdot E_k=0.$$
By Lemma \ref{lem-index}, we have
$$r^3E_i\cdot E_j\cdot E_k\in \mathbb{Z}$$
where $r$ is the numerical Cartier index of $X$.
So we have 
$$E_i\cdot E_j\cdot E_k=0$$
for $i,j,k\gg 1$.
By the projection formula again, this holds for any $i,j,k\in \Z$.
\end{proof}

%

\begin{lemma}\label{lem-nef}
$E_s|_E$ is a nef $\Q$-Cartier divisor on $E$ when $s>0$.
\end{lemma}
\begin{proof}
Let $s>0$.
We may assume $E_s\cap E\neq \emptyset$. 
Write 
$$E_{n+s}\cdot E_n=\sum a_{n,i} C_{n,i}$$ 
where $C_{n,i}$ is an irreducible curve and $a_{n,i}$ is a positive rational number.
By Lemma \ref{lem-index}, $r^2a_{n,i}$ is an integer where $r$ is the numerical Cartier index of $X$.
Note that 
$$f^*(E_{n+s}\cdot E_n)=E_{n+s-1}\cdot E_{n-1}$$ 
when $n\gg 1$.
So we may assume that the number of irreducible components of $E_{n+s}\cdot E_n$ is stable when $n\gg 1$.
Therefore, we may assume 
$$f^*C_{n,i}=t_{n,i}C_{n-1,i}$$ for some integer $t_{n,i}>0$ when $n\gg 1$.
Then for $n\gg 1$ and any $m>0$, we have 
$$r^2a_{n,i}C_{n,i}=(f^m)^*r^2a_{n+m,i}C_{n+m,i}=(t_{n+m,i} t_{n+m-1,i}\cdots t_{n+1,i})\cdot (r^2a_{n+m,i}) C_{n,i}$$
and hence $t_{n+m,i}=1$ for $m\gg 1$.
In particular, $t_{n,i}=1$ when $n\gg 1$.
Then we have 
$$E_{n+s}\cdot C_{n,i}=(f^m)^*E_{n+m+s}\cdot (f^m)^*C_{n+m,i}=(\deg f)^m E_{n+m+s}\cdot C_{n+m,i}$$
when $n\gg 1$.
So we have 
$$\lim_{m\to+\infty}E_{n+m+s}\cdot C_{n+m,i}=0$$
when $n\gg 1$.
Note that $rE_{n+m+s}\cdot C_{n+m,i}$ is an integer for any $n,m$.
Then 
$$E_{n+m+s}\cdot C_{n+m,i}=0$$ 
for $n\gg 1$ and $m\gg 1$.
By the projection formula again, we have
$$E_{n+s}\cdot C_{n,i}=0$$
for any $i$ when $n\gg 1$.
So $E_{n+s}|_{E_n}$ is nef when $n\gg 1$.
Finally, we see that
$$E_s|_E=(f^n|_E)^*E_{n+s}|_{E_n}$$
is nef as desired.
\end{proof}

\begin{proof}[Proof of Theorem \ref{thm-fiitaka1}]
By Lemma \ref{lem-fkappa>0}, $\kappa_f(X,E)>0$.
We show in the following that $\kappa_f(X,E)<3$.
Suppose the contrary.
Then $\sum\limits_{i=1}^k E_i$ is big for some $k>0$.
So 
$$\sum\limits_{i=1}^k E_i\sim_{\Q} A+F$$ 
where $A$ is an ample $\Q$-divisor and $F$ is an effective $\Q$-divisor.
Note that there exists $n>0$ such that $E_{-n}\not\subseteq \Supp F$.
Write 
$$(f^n)_*\sum\limits_{i=1}^k E_i=\sum\limits_{i=n+1}^{n+k} a_iE_i$$ 
where $a_i$ are positive integers.
Then we have
$$\sum\limits_{i=n+1}^{n+k} a_iE_i\sim_{\Q} (f^n)_*A+(f^n)_*F$$
where $(f^n)_*A$ is still ample.
Note that $E\not\subseteq \Supp (f^n)_*F$.
Indeed, if $E\subseteq \Supp (f^n)_*F$, then there is a prime divisor $F_1\subseteq \Supp F$ such that $f^n(F_1)=E$ and hence $F_1\subseteq f^{-n}(E)=E_{-n}$.
So $E_{-n} = F_1 \subseteq \Supp F$, a contradiction.
In particular, 
$$(\sum\limits_{i=n+1}^{n+k} a_iE_i)|_E\sim_{\Q} ((f^n)_*A)|_E+((f^n)_*F)|_E$$ is big.
By Lemma \ref{lem-nef}, $(\sum\limits_{i=n+1}^{n+k} a_iE_i)|_E$ is further nef.
So $(\sum\limits_{i=n+1}^{n+k} a_iE_i)^2\cdot E>0$, contradicting Lemma \ref{lem-ijk}.
\end{proof}

\begin{remark}
We believe Theorem \ref{thm-fiitaka1} can be generalized to higher dimensions provided that Lemma \ref{lem-nef} can be so generalized too.
Moreover, with Lemmas \ref{lem-ijk} and \ref{lem-nef} holding true, our proof of Theorem \ref{thm-fiitaka1} still works even when $f$ is an automorphism of positive entropy. 
This provides a new approach to Lesieutre's \cite[Theorem 1.7]{Les18}.
\end{remark}

\begin{theorem}\label{thm-fiitaka2}
Let $f:X\to X$ be a non-isomorphic surjective endomorphism of a $\Q$-factorial normal projective threefold $X$.
Let $E$ be a prime divisor of $X$ which is not $f^{-1}$-periodic and has $\kappa(X,E)=0$.
Then the following hold.
\begin{enumerate}
\item $\kappa_f(X,E)=2$.
\item Let $\phi_{f,E}:X\dashrightarrow Y$ be the $f$-Iitaka fibration of $E$.
Then $f|_Y$ is an automorphism.
\item $\phi_{f,E}$ is almost well-defined (cf.~Definition \ref{def-awd}).
\end{enumerate}
\end{theorem}

\begin{proof}
By Lemma \ref{lem-fprime}, $E$ is $f$-prime.
Let $\phi_{f,E}:X\dashrightarrow Y$ be an $f$-Iitaka fibration of $E$.
Then $0<\dim Y=\kappa_f(X,E)<3$ by Theorem \ref{thm-fiitaka1}.
Let $E_i:=f^i(E)$.

Consider the following equivariant dynamical systems
$$\xymatrix{
\save[]+<2pc,0pc>*{h \acts} \restore &W \ar[r]^{p_X}\ar[rd]_{p_Y}  & X \ar@{-->}[d]^{\phi_{f,E}}  &\save[]+<-2pc,0pc>*{\racts f} \restore\\
&& Y&\save[]+<-2pc,0pc>*{\racts g} \restore
}$$
where $W$ is the normalization of the graph of $\phi_{f,E}$.
Let $E_i'$ be the strict transform of $E_i$ in $W$ and $C_i:=p_Y(E_i')$.
By Theorem \ref{thm-fibration}, $E_i$ does not dominate $Y$.
So $C_i\subsetneq Y$.

Suppose the contrary $\dim Y=1$.
Note that $p_Y$ has irreducible general fibres.
Then $E_i'$ is some full general fibre of $p_Y$ when $i\gg 1$.
By \cite[Theorem 5.13]{Uen75}, we have
$$0=\kappa(X,E)=\kappa(X,(f^i)^*E_i)=\kappa(X,E_i)=\kappa(W, p_X^*E_i)\ge \kappa(W,E_i')=1$$
when $i\gg 1$, a contradiction.
So $\dim Y=2$ and (1) is proved.

Note that $p_Y$ has connected fibres.
So $C_i=g^{-1}(C_{i+1})$ by \cite[Lemma 7.3]{CMZ20}.
In particular, $C_i$ is $g$-prime.
Note that $p_Y$ is equi-dimensional (cf.~Theorem \ref{thm-fibration}) and $p_Y$ has irreducible and reduced general fibres.
So $E_i'=p_Y^{-1}(C_i)$ when $i\gg 1$.
If $\deg g>1$, then $Y$ is $\Q$-factorial and $\kappa(Y, C_0)>0$ by Lemma \ref{lem-surf-deg=1}.
However, by \cite[Theorem 5.13]{Uen75},
$$\kappa(X,E)=\kappa(X, E_i)=\kappa(W,p_X^*E_i)\ge \kappa(W, p_Y^*C_i)=\kappa(Y,C_i)>0$$
when $i\gg 1$, a contradiction.
So (2) is proved.

\begin{claim}\label{claim-pxp}
Let $P$ be a $p_X$-exceptional prime divisor.
Then $\dim p_X(P)=\dim p_Y(P)=1$.
\end{claim}
After iteration, we may assume each $p_X$-exceptional prime divisor is $h^{-1}$-invariant.
There is an effective exceptional Cartier divisor $F$ such that $-F$ is $p_X$-ample (cf.~\cite[Lemma 2.62]{KM98}).
Write 
$$F=\sum_{q=1}^s F_q \text{ and }  F_q=\sum_{i=1}^{n_q}a_{q,i} F_{q,i}$$ 
where $F_{q,i}$ are $p_X$-exceptional prime divisors with $h^*F_{q,i}=q F_{q,i}$ and $a_{q,i}>0$.

Denote by $\text{Div}_{\Q}(W)$ the space of Weil $\Q$-divisors without modulo any equivalence.
Denote by $\text{CDiv}_{\Q}(W)$ the space of $\Q$-Cartier divisors without modulo any equivalence.
We can naturally regard $\text{CDiv}_{\Q}(W)$ as a subspace of $\text{Div}_{\Q}(W)$.
Let $V$ be the finite dimensional subspace of $\text{Div}_{\Q}(W)$ spanned by the basis $F_{q,i}$.
Then $h^*|_V$ is a diagonal action.
Let $U=V\cap \text{CDiv}_{\Q}(W)$.
Note that $(h^m)^*F/s^m\in U$ for each $m$ and 
$$\lim\limits_{m\to+\infty}(h^m)^*F/s^m=F_s.$$
So $F_s\in U$.
Since $F\in U$, we have $\sum\limits_{q=1}^{s-1} F_q\in U$.
In a similar way, we see $F_q\in U$ for each $q$.

There exists some $q$ such that $P\subseteq \Supp F_q$.
Suppose $p_X(P)$ is a point.
Then $-F|_P$ is ample and 
$$-F_q|_P=-F|_P + (F-F_q)|_P$$ is big.
Moreover,
\begin{equation}\tag{$*$}
(h|_P)^*(-F_q|_P)\equiv q(-F_q|_P).
\end{equation} 
By \cite[Proposition 1.1]{MZ18}, $(h|_P)^*|_{\N^1(P)}$ admits an ample eigenvector with eigenvalue $q$. 
Then $\deg h|_P=q^2$ and hence
$$1<\deg f=\deg h=q(\deg h|_P)=q^3.$$
So we have $q>1$.
Then $h|_P$ and hence $g|_{p_Y(P)}$ are $q$-polarized by \cite[Theorem 3.11]{MZ18}.
Since $\dim p_Y(P)>0$, we get the following contradiction to (2):
$$\deg g\ge \deg g|_{p_Y(P)}>1.$$ 

Suppose $p_Y(P)=Y$.
Then $\deg h|_P=1$ by (2) and hence
$$\deg f=\deg h=q(\deg h|_P)=q>1.$$
We remark that for any eigenvalue $\lambda$ of $(h|_P)^*|_{\NS(P)}$, $1/\lambda$ is an algebriac integer,
so $\lambda$ can not be a rational number $>1$.
Note that $\dim p_X(P)=1$.
Then $-F|_P$ is $p_X|_P$-ample and $-F_q|_P$ is $p_X|_P$-big.
In particular, $-F_q|_P\not\equiv 0$ and hence $q>1$ is an integral eigenvalue of $(h|_P)^*|_{\N^1(P)}$ 
by observing the same ($*$) above. This is absurd as remarked above.
Thus the claim and hence (3) are proved since $\dim Y=2$ by (1).
\end{proof}

\section{Dynamical Iitaka fibration of $R_f$}

In this section, we characterize the $f$-Iitaka fibration of the ramification divisor $R_f$, where $f:X\to X$ is a surjective endomorphism of a $\Q$-factorial normal projective variety $X$.
We shall frequently use the following simple trick.
Observe that
$$R_{f^s}=\sum\limits_{i=0}^{s-1}(f^i)^*R_f$$
holds for any $s>0$.
Then we have $\kappa_f(X,R_f)=\kappa(X, R_{f^s})$ for some $s\gg 1$.

\begin{lemma}\label{lem-reduced}
Consider the equivariant dynamical systems
$$\xymatrix{
f \acts X \ar@{-->}[r]^{\pi} &Y\racts g
}$$
where $\pi$ is a dominant rational map of normal projective varieties with connected fibres.
Suppose $\deg g=1$.
Then $R_f^v=0$ where $R_f^v$ is the $\pi$-vertical part of $R_f$.
\end{lemma}
\begin{proof}
We may assume $\pi$ is a morphism after taking the normalization of the graph of $\pi$.
For an irreducible closed subset $S\in Y$, we denote by $X_S$ the fibre product $X\times_Y S$.
Note that $X$ is integral.
Then the general fibre of $\pi$ is reduced since our base field has characteristic $0$.

Let $Z$ be the closed subset of $Y$ such that $X_y$ is non-reduced exactly when $y\in Z$. 
Note that $\dim Z<\dim Y$. For any $y\in Z$, $f^*X_y=X_{g^{-1}(y)}$ is non-reduced and hence $g^{-1}(y)\in Z$. So $g^{-1}(Z)=Z$ and $\pi^{-1}(Z)$ is $f^{-1}$-invariant.
Iterating $f$, we may assume that irreducible components of $Z$ and $\pi^{-1}(Z)$ are respectively $g^{-1}$- and $f^{-1}$-invariant.

Let $Q\subseteq X$ be a prime divisor such that $f^*Q=rP +$(others) for some prime divisor $P$ and $r>1$.
Suppose the contrary that $\pi(Q) \ne Y$.
Then $Q \subseteq X_{\pi(Q)}$ implies that $f^*X_{\pi(Q)}=X_{g^{-1}(\pi(Q))}$ is non-reduced, hence $g^{-1}(\pi(Q))\subseteq Z$
and $\pi(Q))\subseteq g(Z) = Z$.
Then $\deg g=1$ implies $f^*X_{\pi(Q)}=X_{\pi(Q)}$ and $f^*Q=Q$, a contradiction.
%
\end{proof}

\begin{lemma}\label{lem-fiitaka-rf}
Let $f:X\to X$ be a surjective endomorphism of a normal projective variety $X$.
Suppose $R_f$ is $\Q$-Cartier.
Let $\phi_{f,R_f}:X\dashrightarrow Y$ be the $f$-Iitaka fibration of $R_f$.
Then either $\dim Y=0$ or $\deg f|_Y>1$.
\end{lemma}
\begin{proof}
Suppose $\deg f|_Y=1$.
By Lemma \ref{lem-reduced}, even iterating $f$, every irreducible component of $\Supp R_f$ dominates $Y$.
By the construction of $f$-Iitaka fibraiton, $\dim Y=0$. 
\end{proof}

\begin{corollary}\label{cor-kappa-rf12}
Let $f:X\to X$ be a non-isomorphic surjective endomorphism of a $\Q$-factorial normal projective threefold $X$.
Let $E$ be a prime divisor of $X$ which is not $f^{-1}$-periodic and has $\kappa(X,E)=0$.
Suppose $0<\kappa_f(X,R_f)<3$.
Then $f$ is $\delta$-imprimitive.
\end{corollary}
\begin{proof}
Let $\phi_{f,E}:X\dashrightarrow Y$ be the $f$-Iitaka fibraiton of $E$ and $\phi_{f,R_f}:X\dashrightarrow Z$ the $f$-Iitaka fibraiton of $R_f$.
By Theorem \ref{thm-fiitaka2}, $\dim Y=2$ and $\deg f|_Y=1$.
By Lemma \ref{lem-fiitaka-rf}, $\deg f|_Z>1$.
In particular, $\phi_{f,R_f}$ does not factor through $\phi_{f,E}$.
Then the induced map $X\dashrightarrow Y\times Z$ is generically finite.
By Lemma \ref{lem-imprimitive}, $f$ is $\delta$-imprimitive.
\end{proof}

Instead of simply considering $\kappa_f(X,R_f)=0$, we deal with a more general situation for the purpose of later treatment on the Fano contractions.

\begin{lemma}\label{lem-kappa0}
Consider the equivariant dynamical systems
$$\xymatrix{
\widetilde{f} \acts \widetilde{X} \ar[r]^{\pi} &X\racts f
}$$
with the following settings:
\begin{enumerate}
\item  $\pi$ is a finite surjective morphism of normal projective varieties with $\deg \pi=2$.
\item  $X$ is smooth and $\delta_f>1$.
\item  $P$ is a prime divisor with $f^*P=\delta_f P$.
\item  $\pi^*P=P_1+P_2$ has two irreducible components.
\end{enumerate}
Then either $\delta_{f|_P}= \delta_f$, or $\kappa(X, P)>0$ and $f$ is $\delta$-imprimitive.
\end{lemma}
\begin{proof}

We may assume $\delta_{f|_P}\neq \delta_f=:\delta$.
Then $\delta_{f|_P}< \delta_f$.
Note that $\delta=\delta_f=\delta_{\widetilde{f}}>1$ is an integer by (3).
After iteration, we may assume $\widetilde{f}^*P_i=\delta P_i$ for $i=1,2$.
Note that $\delta_{\widetilde{f}|_{P_1}}=\delta_{f|_P}$.
Then $\delta$ is not an eigenvalue of $(\widetilde{f}|_{P_1})^*|_{\N^1(P_1)}$.

\begin{claim}\label{claim-qfactorial}
$P_1$ and $P_2$ are $\Q$-Cartier in codimension $2$.
\end{claim}
Note that $R_{\pi}=\pi^*B_{\pi}/2$ is a $\Q$-Cartier reduced divisor.
Then $K_{\widetilde{X}}=\pi^*K_X+R_{\pi}$ is $\Q$-Cartier.
Let $Z$ be the non-lc locus of $\widetilde{X}$ which has codimension at least $2$.
Note that $\deg f\ge 2$ since $P\subseteq B_f$.
Then $\widetilde{f}^{-1}(Z)=Z$ and each irreducible component of $Z$ is not contained in $P_1\cup P_2$ by \cite[Theorem 1.4]{BH14}.
In particular, $Z\cap P_i$ has codimension at least $3$ in $\widetilde{X}$.
Let $U:=\widetilde{X}\backslash Z$.
Then $U$ is log canonical and further $U$ is klt arround $(P_1\cup P_2)\cap U$ by \cite[Proposition 2.2]{Zha14}.
It is known that dlt and hence klt varieties are $\Q$-factorial in codimension $2$ (cf.~\cite[Proposition 9.1]{GKKP11}).
So the claim is proved.

%

\begin{claim}\label{claim-p1p2}
$R_{\pi}\cap P_1=R_{\pi}\cap P_2=\pi^{-1}(B_{\pi}\cap P)\subseteq P_1\cap P_2$.
\end{claim}

It is clear that
$R_{\pi}\cap P_1\subseteq \pi^{-1}(B_{\pi}\cap P)$.
Conversely, for any $x\in B_{\pi}\cap P$, we have $\pi^{-1}(x)\cap P_i\neq \emptyset$ since $ \pi|_{P_i}:P_i\to P$ is surjective.
Note that $x\in B_{\pi}$ and $\deg \pi=2$.
Then $\pi^{-1}(x)$ is a single point in $R_{\pi}$.
So we have $\pi^{-1}(x)\subseteq R_{\pi}\cap P_i$ and the claim is proved.

\vskip 2mm
We continue the proof of Lemma \ref{lem-kappa0}.

Suppose $B_{\pi}\cap P=\emptyset$.
By the purity of branch locus, $\pi$ is \'etale near $P$.
Then $\widetilde{X}$ is smooth near $P_1\cup P_2$.
In particular, $P_1$ and $P_2$ are Cartier.
Note that $(\widetilde{f}|_{P_i})^*(P_i|_{P_j})=\delta P_i|_{P_j}$.
So $P_i|_{P_j}\equiv 0$.
Then $P_1$ and $P_2$ (and hence $P$) are nef and hence $P_1\equiv aP_2$ for some rational number $a>0$ by Lemma \ref{lem-hodge}.
Note that $r(P_1-aP_2)\in \Pic^0(X)$ for some integer $r>0$ and $\widetilde{f}^*r(P_1-aP_2)=\delta\cdot r(P_1-aP_2)$.
So $P_1-aP_2\sim_{\Q} 0$ by \cite[Proposition 6.3]{MMS+22}.
In particular, $\kappa(X, P)=\kappa(\widetilde{X}, P_1+P_2)>0$.
Note that $P|_P\equiv 0$.
So $0 < \kappa_f(X, P)=\kappa(X, P)<\dim X$.
Let $\phi:X\dashrightarrow Y$ be the $f$-Iitaka fibraiton of $P$.
Then $f|_Y$ is $\delta$-polarized by Theorem \ref{thm-fiitaka-polarized}.
In particular, $f$ is $\delta$-imprimitive.

Suppose $B_{\pi}\cap P\neq \emptyset$.
By Claim \ref{claim-p1p2} and since $R_{\pi}$ is $\Q$-Cartier, $R_{\pi}\cap P_1$ is a non-empty pure codimension $2$ closed subset in $\widetilde{X}$.
Then $P_1\cap P_2$ contains an irreducible component $S$ of codimension $2$ in $\widetilde{X}$.
After iteration, we may assume $\widetilde{f}^{-1}(S)=S$.
By Lemma \ref{lem-d1-weil}, $(\widetilde{f}|_{P_1})^*S=aS$ for some integer $a\le\delta_{\widetilde{f}|_{P_1}}<\delta$.
By Claim \ref{claim-qfactorial}, we can find some open subset $V\subseteq X$ with complement of codimension $3$, such that $P_2$ is $\Q$-Cartier on $V$ and $f^{-1}(V)$.
Then $D:=P_2|_{P_1\cap V}$ (resp. $D':=P_2|_{P_1\cap f^{-1}(V)}$) is an effective $\Q$-Cartier divisor with $S\cap V$ (resp. $S\cap f^{-1}(V)$) being a component of $\Supp D$ (resp. $\Supp D'$).
Let $g:=\widetilde{f}|_{P_1}$ and $h:=g|_{P_1\cap f^{-1}(V)}:P_1\cap f^{-1}(V)\to P_1\cap V$.
We have $h^*D=(f^*P_2)|_{P_1\cap f^{-1}(V)}=\delta D'$.
Then $h^*(S\cap V)=\delta (S\cap f^{-1}(V))$ and hence $g^*S=\delta S$, a contradiction.
\end{proof}

\begin{lemma}\label{lem-Rf=0}
Consider the equivariant dynamical systems of normal projective varieties
$$\xymatrix{
f \acts X \ar@{-->}[r]^{\phi} &Y\racts g
}$$
with the following settings:
\begin{enumerate}
\item  $X$ is smooth with $\dim X=n$ and $\dim Y=n-1>0$.
\item  $\phi$ is an almost well-defined dominant rational map.
\item  $f$ is $\delta$-primitive.
\item  $\kappa_f(X,R_f^h)=0$ where $R_f^h$ is the $\phi$-horizontal part of $R_f$.
\end{enumerate}
Then $R_f^h=0$ and the general fibre of $\phi$ is an elliptic curve.
\end{lemma}
\begin{proof}
Let $X_y$ be a general fibre of $\phi$.
Since $\phi$ is almost well-defined, we may assume $X_y$ and $X_{g(y)}$ are smooth projective curves with the same genus.
Since $f$ is $\delta$-primitive, $\delta_f>\delta_g \ge 1$. Then, $\dim X_y=1$ and the product formula imply
$\deg f|_{X_y} = \delta_f >1$.
By the adjunction formula and the ramification divisor formula, we have
$$R_{f|_{X_y}}=R_f^h|_{X_y}.$$
If $R_f^h=0$, then $f|_{X_y}$ is \'etale and hence $X_y$ is an elliptic curve.

Suppose the contrary that $R_f^h\neq 0$.
Since $\kappa_f(X,R_f^h)=0$, each irreducible component $P_i$ of $P:=\Supp R_f^h$ is $f^{-1}$-invariant after iteration by Lemma \ref{lem-fkappa>0}.

Suppose further that $P$ has at least $2$ irreducible components $P_1$ and $P_2$.
Write $f^*P_i=a_i P_i$.
Then we have 
$$a_i=\deg f/\deg f|_{P_i}=\deg f/ \deg g=\deg f|_{X_y}=\delta_{f|_{\phi}}=\delta_f,$$
where the second equality is due to $P_i$ dominating $Y$ and the last equality is by the dynamical product formula and $\delta_f>\delta_g$ according to (3).
Since $\phi|_{P_i}:P_i\dashrightarrow Y$ is generically finite, $\delta_{f|_{P_i}} = \delta_g<\delta_f$.
So $(f|_{P_i})^*P_j|_{P_i}=\delta_f P_j|_{P_i}$ implies $P_j|_{P_i}\equiv 0$.
Hence $P_i$ is nef and $P_1\equiv tP_2$ for some rational number $t>0$ by Lemma \ref{lem-hodge}.
Note that $s(P_1-tP_2)\in \Pic^0(X)$ and $f^*s(P_1-tP_2)=\delta_f\cdot s(P_1-tP_2)$.
Since $\delta_f>1$, we have $P_1-tP_2\sim_{\Q} 0$ by \cite[Proposition 6.3]{MMS+22}, contradicting (4).
Thus $P$ is irreducible.

Note that $R_{f|_{X_y}}=R_f^h|_{X_y}$.
So $f|_{X_y}$ is totally ramified.
Hence $X_y\cong \mathbb{P}^1$ and $\deg \phi|_P=2$.
We construct an $f$-equivariant double cover of $X$ by the following commutative diagram:
$$\xymatrix{
\widetilde{P}\ar[d]_{\widetilde{\phi|_P}}&\widetilde{W}\ar[r]^{p_{\widetilde{X}}}\ar[d]^{p_W}\ar[l]_{p_{\widetilde{P}}}&\widetilde{X}\ar[d]^{\pi}\\
Y&W\ar[l]^{\phi_Y}\ar[r]_{\phi_X}&X
}$$
where $W$ is the normalization of the graph of $\phi$ with two projections $\phi_X$ and $\phi_Y$, $\widetilde{\phi|_P}:\widetilde{P}\to Y$ is the normalization of $Y$ in the function field $k(P)$, $\widetilde{W}$ is the main component of the normalization of $W\times_Y \widetilde{P}$, and $\pi:\widetilde{X}\to X$ is obtained by taking the Stein factorization of $\widetilde{W}\to X$.
Note that this diagram is $f$-equivariant by \cite[Lemma 5.2]{CMZ20}.
Equivalently, $\widetilde{X}$ is the normalization of $X$ in $k(X)\otimes_{k(Y)} k(P)$.
Note that $\deg \pi=\deg \widetilde{\phi|_P}=2$.
The natural embedding of $k(P)\hookrightarrow k(P)\otimes_{k(Y)}k(P)$ implies that $\pi^*P=\pi^{-1}(P)=P_1+P_2$ has two irreducible components.
By Lemma \ref{lem-kappa0} and since $f$ is $\delta$-primitive, $\delta_f=\delta_{f|_P}=\delta_g$, a contradiction.
\end{proof}

\begin{corollary}\label{cor-kappa-rf0}
Let $f:X\to X$ be a non-isomorphic surjective endomorphism of a smooth projective threefold $X$.
Let $E$ be a prime divisor of $X$ which is not $f^{-1}$-periodic and has $\kappa(X,E)=0$.
Suppose $\kappa_f(X,R_f)=0$.
Then $f$ is either $\delta$-imprimitive or \'etale.
\end{corollary}
\begin{proof}
Let $\phi:X\dashrightarrow  Y$ be the $f$-Iitaka fibration of $E$.
By Theorem \ref{thm-fiitaka2}, $\phi$ is almost well-defined, $\dim Y=2$ and $\deg f|_Y=1$.
Suppose $f$ is $\delta$-primitive.
Then $R_f = 0$ by
Lemmas \ref{lem-reduced} and \ref{lem-Rf=0}.
Hence $f$ is \'etale by the purity of branch locus. 
\end{proof}

When $\kappa_f(X,R_f)=\dim X$, or equivalently when $R_f$ is big (after iterating $f$), we propose the following general question.
Note that if Question \ref{que-big} is answered affirmatively, we can apply the int-amplified EMMP theory when $R_f$ is big.

\begin{question}\label{que-big}
Let $f:X\to X$ be a surjective endomorphism of a normal projective variety $X$.
Suppose $R_f$ is a big Weil divisor.
Is $f$ int-amplified?
\end{question}

For the purpose of this paper, it suffices for us to treat the following special cases.

\begin{theorem}\label{thm-E-Rf-notbig}
Let $f:X\to X$ be a non-isomorphic surjective endomorphism of a $\Q$-factorial lc projective threefold $X$.
Suppose $X$ admits a divisorial contraction $\pi_0:X\to X_0$ of some $K_X$-negative extremal ray with the exceptional (prime) divisor $E$ not being $f^{-1}$-periodic.
Then $\kappa_f(X,R_f)<3$.
\end{theorem}

\begin{proof}
Suppose the contrary that $\kappa_f(X,R_f)=3$.
Then $R_f$ and hence $B_f=f(\Supp R_f)$ are big after iteration.
Let $\phi:X\dashrightarrow Y$ be the $f$-Iitaka fibraiton of $E$ and $g:=f|_Y$.
Note that $\kappa(X,E)=0$.
By Theorem \ref{thm-fiitaka2}, $\dim Y=2$, $\deg g=1$, and $\phi$ is almost well-defined.
Let $E_i:=f^i(E)$ and $C_i$ the image of $E_i$ in $Y$ with $i\ge 0$.
Note that $C_i$ is an irreducible curve by Theorem \ref{thm-fibration}, and $\phi|_{E_i}:E_i\dashrightarrow C_i$ is almost well-defined when $i\gg 1$.

Let $\ell_0$ be an irreducible curve contracted by $\pi_0$.
Let $\ell_i:=f^i(\ell_0)$ with $i\ge 0$.
Then $E_i$ is covered by the curves in the extremal ray $\mathbb{R}_{\ge 0} [\ell_i]$.
Since $R_f|_{E_i}$ is big when $i\gg 1$,
there exists some $s>0$ such that $R_f\cdot \ell_i>0$ when $i\ge s$.
Let $r$ be the numerical Cartier index of $X$.
Then $r(f^j)^*R_f\cdot \ell_s=rR_f\cdot (f^j)_*\ell_s$
is a positive integer for any $j>0$.
Therefore,
$$\lim\limits_{n\to +\infty} R_{f^n}\cdot \ell_s=\sum_{i=0}^{+\infty}(f^j)^*R_f\cdot \ell_s=+\infty.$$
Take $n\gg 1$.
By the ramification divisor formula,  we may assume
$$(f^n)^*K_X\cdot \ell_s=K_X\cdot \ell_s-R_{f^n}\cdot \ell_s<0.$$ 
By the projection formula, we have
$K_X\cdot \ell_{s+n}<0$.
By the cone theorem (cf.~\cite[Theorem 3.7]{KM98}), there exists a divisorial contraction 
$\pi_i:X\to X_i$ of $\ell_i$ with $E_i$ the excetpional divisor when $i\gg 1$.
We may replace $E$ by $E_{n}$ with $n\gg 1$ and assume: 
{\it $K_X\cdot \ell_i<0$,
$f^*E_{i+1} = E_i$ (and hence $\deg f|_{E_i} = \deg f$),
$B_f|_{E_i}$ is a big divisor,
$E_i\not\subseteq B_f\cup \Sing(X)$, and 
$\phi|_{E_i}:E_i\dashrightarrow C_i$ is almost well-defined
for any $i\ge 0$.}

Let $B_i:=\pi_i(E_i)$. If $B_i$ is a point, then $E_i^3 < 0$ by \cite[Lemma 2.62]{KM98}, contradicting
Lemma \ref{lem-ijk}. So $B_i$ is a curve.
By the rigidity lemma (cf.~\cite[Lemma 1.15]{Deb01}), we have the following commutative diagram
$$\xymatrix{
X\ar[r]^{\pi_i}\ar[d]_f &X_i\ar[d]^{f_i}\\
X\ar[r]^{\pi_{i+1}} &X_{i+1}
}$$
where $f_i$ is finite surjective with $f_i^{-1}(B_{i+1})=B_i$. 
Note that $-E_i$ is $\pi_i$-ample.
So $E_i\cdot \ell_i<0$.
On the other hand, $E_i$ is disjoint with the general fibre of $\phi$ and $\phi|_{E_i}:E_i\dashrightarrow C_i$ is almost well-defined when $i\gg 1$.
Then the induced map $\sigma_i:E_i\dashrightarrow B_i\times C_i$ is generically finite.
Consider the following commutative diagram
$$\xymatrix{
E_i\ar@{-->}[rr]^{\sigma_i}\ar[d]_{f|_{E_i}} && B_i\times C_i\ar[d]^{f_i|_{B_i}\times g|_{C_i}}\\
E_{i+1}\ar@{-->}[rr]_{\sigma_{i+1}} && B_{i+1}\times C_{i+1}
}$$
where $\deg f|_{E_i}=\deg f$ and $\deg (f_i|_{B_i}\times g|_{C_i})=\deg f_i|_{B_i}\le \deg f$.
So $\deg \sigma_i\ge \deg \sigma_{i+1}$.
We may further assume $\deg \sigma_i=\sigma_{i+1}$ for any $i\ge 0$ after replacing $E$ by $E_n$ with $n\gg 1$.
Now $\deg g|_{C_0}=1$ implies
$\deg f_0|_{B_0}=\deg f|_{E_0}=\deg f$ by the above diagram.

Let $v\in B_1$ be a general point such that $\sharp f_0^{-1}(v)=\deg f$. This is possible because the base field has characteristic $0$.
Note that $$f^{-1}(\pi_1^{-1}(v))=\bigcup_{u\in f_0^{-1}(v)}\pi_0^{-1}(u)$$
is a disjoint union.
Then $\sharp f^{-1}(x)=\deg f$ for any $x\in \pi_1^{-1}(v)$.
So $\pi_1^{-1}(v)\cap B_f=\emptyset$ and hence $B_f\cdot \ell_1=0$, contradicting $B_f|_{E_1}$ being a big divisor.
\end{proof}

\begin{theorem}\label{thm-fano-iamp}
Consider the following equivariant dynamical systems of $\Q$-factorial lc projective varieties
$$\xymatrix{
f \acts X\ar[r]^{\tau} &Y\racts g
}$$
where $\dim Y=1$, and $\tau$ is a Fano contraction of some $K_X$-negative extremal ray. 
Suppose $R_f$ is big.
Then $f$ is int-amplified and $\deg g>1$.
\end{theorem}
\begin{proof}
Let $F$ be a general fibre of $\tau$ and $D$ another pseudo-effective divisor extremal in $\PE^1(X)$.
After iteration, we may write $f^*F\equiv aF$ and $f^*D\equiv bD$ where $a$ and $b$ are positive integers.
Write $K_X\equiv sF+tD$.
Then the bigness of 
$$R_f=K_X-f^*K_X\equiv s(1-a)F+t(1-b)D$$ 
implies
$a \ne 1 \ne b$, i.e., $a>1$ and $b>1$.
By \cite[Theorem 1.1]{Men20}, $f$ is int-amplified.
\end{proof}

\begin{theorem}\label{thm-rf-big}
Consider the equivariant dynamical systems of $\Q$-factorial lc projective varieties
$$\xymatrix{
f \acts X\ar[r]^{\tau} &Y\racts g
}$$
satisfying the following 
\begin{enumerate}
\item $\dim X=n$ and $\dim Y=m>0$,
\item $\tau$ is a Fano contraction of some $K_X$-negative extremal ray,
\item $\tau$ is the only $K_X$-negative extremal contraction, and
\item $\delta_f>\delta_g$.
\end{enumerate}
Then $R_f$ is not big.
\end{theorem}
\begin{proof}
By Perron-Frobenius theorem, $f^*D \equiv \delta_f D$ for some nonzero nef $\R$-Cartier divisor $D$.
If $D$ is $\tau$-trivial, then the cone theorem \cite[Theorem 3.7]{KM98} implies that $D$ is pulled back from $Y$,
contradicting that $\delta_g < \delta_f$. Thus $D$ is $\tau$-ample. Hence $D^{n-m}\not\equiv_w 0$. 
Since $\N^1(X)/\pi^*\N^1(Y)$ is 1-dimensional and $f^*|_{\N^1(X)/\pi^*\N^1(Y)} = q \id$ for some integer $q>0$,
$ \delta_f=q$.
We may then choose $D$ to be Cartier, see also the second paragraph of the proof of \cite[ Proposition 9.2]{MZ22}.
Since $f$ is not polarized (or else $\delta_f = \delta_g$, absurd), $D$ is not big (cf.~\cite[Proposition 1.1]{MZ18}).
So $D^n=0$.
Let $t$ be the integer such that $D^t\not\equiv_w 0$ and $D^{t+1}\equiv_w 0$.
Fix an ample Cartier divisor $H$ of $X$.

Let $a$ be a positive rational number such that $K_X+aD$ is $\tau$-trivial.
Since $D$ is nef, any $(K_X+aD)$-negative extremal ray is also $K_X$-negative.
Then $K_X+aD$ is nef by the assumption (3).
Note that 
$$0\le (f^s)^*(K_X+aD)\cdot D^t\cdot H^{n-t-1}=(K_X-R_{f^s})\cdot D^t\cdot H^{n-t-1}$$
for any $s>0$.
Suppose the contrary that $R_f$ is big.
Then $R_f=(1/m)H+E$ for some effective $\Q$-Cartier divisor $E$ and $m>0$.
Note that
$$R_{f^s}\cdot D^t\cdot H^{n-t-1}\ge (1/m)\sum_{i=0}^{s-1} (f^i)^*H\cdot D^t\cdot H^{n-t-1}\ge s/m.$$
Then $(K_X-R_{f^s})\cdot D^t\cdot H^{n-t-1}<0$ when $s\gg 1$, contradicting the early inequality.
\end{proof}

\begin{corollary}\label{cor-fano-n-1}
Consider the equivariant dynamical systems of normal projective varieties
$$\xymatrix{
\save[]+<2pc,0pc>*{f \acts} \restore &X \ar@{-->}[r]^{\sigma}  & X' \ar[d]^{\tau}  &\save[]+<-1.93pc,0pc>*{\racts f'} \restore\\
&& Y&\save[]+<-2pc,0pc>*{\racts g} \restore
}$$
satisfying the following 
\begin{enumerate}
\item $X$ is smooth of dimension $n$, $\sigma$ is birational and $\sigma^{-1}$ contracts no divisors,
\item $X'$ is $\Q$-factorial lc and $\dim Y=n-1>0$,
\item $\tau$ is a Fano contraction, and
\item $\tau$ is the only $K_X$-negative extremal contraction.
\end{enumerate}
Then $f$ is $\delta$-imprimitive.
\end{corollary}
\begin{proof}
Suppose the contrary that $f$ is $\delta$-primitive.
Then $f'$ is $\delta$-primitive too, and $\delta_{f'}>\delta_g$.
Let  $R_{f'}^h$ be the $\tau$-horizontal part of $R_f$.
If $\kappa_{f'}(X, R_{f'}^h)=n$,
then $R_{f'}^h$ and hence $R_{f'}$ are big after iteration, a contradiction to Theorem \ref{thm-rf-big}.
Let $\phi:X'\dashrightarrow Z$ be the $f'$-Iitaka fibration of $R_{f'}^h$.
If $\dim Z=\kappa_{f'}(X, R_{f'}^h)>0$, then $\phi$ does not factor through $\tau$ and hence
the induced map $X'\dashrightarrow Y\times Z$ is generically finite by noting that $\dim X'=\dim Y+1$.
By Lemma \ref{lem-imprimitive}, $f'$ is $\delta$-imprimitive, a contradiction.
So $\kappa_{f'}(X, R_{f'}^h)=0$.
Let $R_f^h$ be the $\tau\circ\sigma$-horizontal part of $R_f$.
Since $\sigma^{-1}$ contracts no divisors, we have $R_{f'}=\sigma_*R_f$.
So $\kappa_f(X,R_f^h)=0$ by Lemma \ref{lem-pushforward-fiitka}.
Since $\dim X=\dim Y+1$ and  $\sigma^{-1}$ contracts no divisors, $\tau\circ\sigma$ is almost well-defined.
By Lemma \ref{lem-Rf=0}, the general fibre of $\tau\circ\sigma$ (isomorphic to the general fibre of $\tau$) is an elliptic curve, contradicting that $\tau$ is a Fano contraction.
\end{proof}

\section{Equivariant flips}
In this section, we prove the equivariancy for flips in the birational MMP starting from a smooth projective threefold $X$.
Theorem \ref{thm-emmp-flip} is our main result.
As usual, our MMP of $X$ has no boundary part and involves only $K_X$-negative extremal rays.

We recall the following useful result from the proof of Claim \ref{claim-pxp}.
\begin{lemma}\label{lem-eq-cartier}
Consider the equivariant dynamical systems of normal projective varieties
$$\xymatrix{
h \acts W \ar[r]^{\sigma} &X\racts f
}$$
where $Y$ is $\Q$-factorial and $\sigma$ is birational.
Suppose each exceptional prime divisor of $\sigma$ is $h^{-1}$-invariant.
Then there is an effective Cartier $\sigma$-exceptional divisor $E=\sum\limits_q E_q$ with $\Supp E=\Exc(\sigma)$ such that $-E$ is $\sigma$-ample, $E_q$ is Cartier, and $h^*E_q=qE_q$.
Moreover, $-E_q|_P$ is $\sigma|_P$-big for any prime divisor $P\subseteq \Supp E_q$.
In particular, $h|_P$ is $q$-polarized when $\sigma(P)$ is a point.
\end{lemma}

\begin{definition}
Let $f:X\to X$ be a surjective endomorphism of a projective variety.
We say $f$ is {\it $q$-amplified} if all the eigenvalues of $f^*|_{\N^1(X)}$ have the same modulos $q>1$.
If $f$ is $q$-amplified (resp.~$q$-polarized), then $f$ is int-amplified (resp.~$q$-amplified).

Like $q$-polarized and int-amplified endomorphisms, being $q$-amplified is preserved under generically finite lifting and dominant descending.
\end{definition}

\begin{lemma}\label{lem-qamplified}
Let $f:S\to S$ be a surjective endomorphism of a projective surface $S$ with $\deg f=q^2$ for some $q>1$
Let $D\in \N^1(S)\backslash\{0\}$ such that $D$ is pseudo-effective and $f^*D\equiv qD$.
Then $f$ is $q$-amplified.
\end{lemma}
\begin{proof}
We may assume $S$ is normal.
Suppose the contrary that $f$ is not $q$-amplified.
Then $\delta_f>\iota_f$ (cf.~Definition \ref{def-d1}).
By Lemma \ref{lem-surf-deg}, $q^2=\deg f=\delta_f\cdot \iota_f$.
So $\delta_f > q > \iota_f$.

Suppose $K_{S}$ is not pseudo-effective.
By \cite[Theorem 5.4]{MZ22}, we may assume $\rho(S)=2$.
After iteration, $f^*|_{\N^1(S)}$ is a diagonal action with at most two different eigenvalues, a contradiction.

Suppose $K_{S}$ is pseudo-effective.
Then  $f$ is quasi-\'etale, and $S$ is a quasi-\'etale quotient of either an abelian surface or a product of an elliptic curve $E$ and a smooth projective curve $T$ of genus $\ge 2$, see \cite[Theorem 5.1]{MZ22}.
Suppose $S$ admits a negative curve $C$.
Then the latter case occurs: $S\cong E\times T$.
Note that $f^{-1}(C)=C$ after iteration (cf.~\cite[Lemma 4.3]{MZ22}) and hence $f^*C=qC$ with $q=\sqrt{\deg f}>1$ by the projection formula.
However, $C\subseteq B_f$ and $f$ is not quasi-\'etale, a contradiction.
In particular, $\Nef(S)=\PE^1(S)$.
Then $D$ is further nef.
Write $f^*D_+\equiv \delta_f D_+$ for some $D_+\in \Nef(S)\backslash\{0\}$.
We have $D\cdot D_+\neq 0$ by \cite[Lemma 4.1]{MZg22}.
By the projection formula, $\deg f=\delta_f\cdot q$, a contradiction. 
\end{proof}

\begin{lemma}[Connecting Lemma]\label{lem-connection}
Consider the equivariant dynamical systems of normal projective threefolds
$$\xymatrix{
h \acts W \ar[r]^{\sigma} &X\racts f
}$$
where $X$ is klt and $\sigma$ is birational.
Let $Q$ be an $h^{-1}$-invariant prime divisor of $W$ such that $h|_Q$ is $q$-amplified with $q^3=\deg h\ge 2$.
Then $h|_P$ is $q$-amplified and $h^*P=qP$ for any $h^{-1}$-invarint prime divisor $P$ with $P\cap Q\neq\emptyset$. 
\end{lemma}

\begin{proof}
After iteration, we may assume each exceptional prime divisor of $\sigma$ is $h^{-1}$-invariant.
Set $q:=(\deg h)^{1/3}$.
By the assumption and Lemma \ref{lem-surf-deg}, $\deg h|_Q=q^2$ and $h^*Q=qQ$.
Write $h^*P=rP$ for some integer $r>0$.
Note that $\deg h|_P=q^2$ if $r=q$.

\textbf{Case $\dim \sigma(Q)=0$ and $\dim \sigma(P)=2$}.

Note that 
$$\deg f|_{\sigma(P)}=\deg h|_P=\frac{\deg h}{r}.$$
So we have
$$f^*\sigma(P)=r\sigma(P)$$ and hence 
$$h^*(\sigma^*\sigma(P))= r\sigma^*\sigma(P).$$
Note that $Q\subseteq \Supp \sigma^*\sigma(P)$.
So $r=q$ and $\deg h|_P=q^2$.
Since $Q$ is $\sigma$-exceptional, $Q\subseteq \Supp E_q$ by Lemma \ref{lem-eq-cartier}.
Then $E_q|_P$ is a non-zero effective Cartier divisor and
$$(h|_P)^*E_q|_P\sim qE_q|_P.$$
By Lemma \ref{lem-qamplified}, $h|_P$ is $q$-amplified.

\textbf{Case $\dim \sigma(Q)=1$ and $\dim \sigma(P)=2$}.

If $\sigma(Q)\subseteq\sigma(P)$, then we are done for the same reason as in the previous case.
Suppose $\sigma(Q)\not\subseteq\sigma(P)$.
Then $\sigma(P)|_{\sigma(Q)}$ is a non-zero effective $\Q$-Cartier divisor and
$$(f|_{\sigma(Q)})^*\sigma(P)|_{\sigma(Q)}=r\sigma(P)|_{\sigma(Q)}.$$
Note that $h|_Q$ and hence $f|_{\sigma(Q)}$ are $q$-amplified.
So $r=q$.
Since $Q\subseteq \Supp E_q$, we have 
$$(h|_P)^*E_q|_P\sim qE_q|_P$$
where $E_q|_P$ is a non-zero effective Cartier divisor.
By Lemma \ref{lem-qamplified}, $h|_P$ is $q$-amplified .

\textbf{Case $\dim \sigma(Q)=2$ and $\dim \sigma(P)=2$}.

If $\sigma(Q)=\sigma(P)$, then $P=Q$ and we are done.
Suppose $\sigma(Q)\neq \sigma(P)$.
By looking at $\sigma(P)|_{\sigma(Q)}$, we have $r=q$ by a similar reason.
Note that $\sigma(Q)|_{\sigma(P)}$ is a non-zero effective $\Q$-Catier divisor and 
$$(f|_{\sigma(P)})^*\sigma(Q)|_{\sigma(P)}=q\sigma(Q)|_{\sigma(P)}.$$
By Lemma \ref{lem-qamplified}, $f|_{\sigma(P)}$ and hence $h|_P$ are $q$-amplified.

\textbf{Case $\dim \sigma(P)\le 1$}.

By Lemma \ref{lem-eq-cartier}, $P\subseteq \Supp E_r$.
If $Q\subseteq \Supp E_r$, then $r=q$.
If $Q\not\subseteq \Supp E_r$, then $E_r|_Q$ is a non-zero effective Cartier divisor and
$$(h|_Q)^*E_r|_Q\sim rE_r|_Q.$$
So $r=q$ all the time.
Note that  
$$(h|_P)^*E_q|_P\sim qE_q|_P$$
where $E_q|_P$ is $\sigma|_P$-big by Lemma \ref{lem-eq-cartier}.
If $\dim \sigma(P)=0$, then $-E_q|_P$ is big.
So $h|_P$ is $q$-polarized (cf.~\cite[Proposition 1.1]{MZ18}) and hence $q$-amplified.
If $\dim \sigma(P)=1$, then $-E_q|_P\cdot \ell>0$ where $\ell$ is a general fibre of $\sigma|_P$.
Note that $f|_{\sigma(P)}$ is $a$-polarized with $a=\deg f|_{\sigma(P)}$.
Then $(f|_{\sigma(P)})^*\ell\equiv a\ell$ and $\deg h|_P=q\cdot a=q^2$ by the projection formula.
So $a=q$.
Note that $\ell$ is effective on $P$.
By Lemma \ref{lem-qamplified}, $h|_P$ is $q$-amplified.
\end{proof}

Lemma \ref{lem-connected} tells the influence of earlier exceptional divisors on the later singular locus during the birational MMP starting from a smooth projective threefold.
We begin with:

\begin{lemma}\label{lem-ben}
Let $X$ be a $\Q$-Gorenstein normal projective threefold with canonical singularities.
Let $\pi:X\to Y$ be the flipping contraction of some $K_X$-negative curve $C$.
Then $C\cap \Sing(X)\neq \emptyset$.
\end{lemma}
\begin{proof}
If $C\cap  \Sing(X)= \emptyset$, then $K_X\cdot C\le -1$.
However, this contradicts \cite[Theorem 0]{Ben85}.
\end{proof}

The lemma below is essentially due to \cite{Mor82}.
We give a proof for the convenience.

\begin{lemma}\label{lem-mori}
Let $X$ be a $\Q$-Gorenstein normal projective threefold and
$\pi:X\to Y$ a divisorial contraction.
Suppose the $\pi$-exceptional prime divisor $D$ is contracted to a curve $S$ in $Y$.
Then $\pi^{-1}(y)\cong \mathbb{P}^1$ and $y\not\in \Sing(Y)$ for any $y\in S$ with  $\pi^{-1}(y)\cap \Sing(X)=\emptyset$.
\end{lemma}
\begin{proof}
Let $C=\pi^{-1}(y)$ with $y\in S$ and $C\cap \Sing(X)=\emptyset$.
Denote by $\mathcal{I}$ the ideal sheaf of $C$ in $X$.
Let $C_n$ be the closed subscheme of $X$ defined by $\mathcal{I}^n$.

By the same argument of \cite[Lemmas 3.20]{Mor82},
we have $\chi(\mathcal{O}_{C'})\ge 0$ for any closed subscheme $C'$ of $X$ with $(C')_{\red}=C$.
In this step, we do not need $C\cap \Sing(X)=\emptyset$.
Applying the same argument of \cite[(3.38)]{Mor82}, we further have $C\cong \mathbb{P}^1$
by noticing that 
$$\chi(\Omega_X^1\otimes\mathcal{O}_C)=K_X\cdot C<0$$ still holds.
By the same discussion as in \cite[(3.39)]{Mor82}, we have 
$$\mathcal{I}/\mathcal{I}^2\cong \mathcal{O}_C\oplus\mathcal{O}_C(1).$$
It is easy to compute that
$$\mathcal{I}^n/\mathcal{I}^{n+1}\cong \Sym^n(\mathcal{I}/\mathcal{I}^2)\cong \bigoplus_{i=0}^{n}\mathcal{O}_C(i)$$
and hence 
$$h^0(C_n,\mathcal{I}^n/\mathcal{I}^{n+1})=\frac{(n+1)(n+2)}{2}$$
Using the exact sequence 
$$0\to \mathcal{I}^n/\mathcal{I}^{n+1}\to \mathcal{O}_{C_n}\to \mathcal{O}_{C_{n+1}}\to 0,$$
we have 
$$H^0(C_n,\mathcal{O}_{C_n})\cong K[x,y,z]/(x,y,z)^n$$
by induction, where $K$ is the base field.

Let $X_n$ be the closed subscheme of $X$ defined by $m_y^n\cdot\mathcal{O}_X$ where $m_y$ is the maximal ideal of $y$.
Let $\hat{\mathcal{O}}_y$ be the completion of the stalk $\mathcal{O}_y$.
By the theorem of formal functions, 
$$\hat{\mathcal{O}}_y\cong\lim_{\longleftarrow} H^0(X_n,\mathcal{O}_{X_n}).$$
Note that the sequence of ideals $m_y^n\cdot\mathcal{O}_X$ is cofinal with the sequence of ideals $\mathcal{I}^n$.
So
$$\hat{\mathcal{O}}_y\cong K[[x,y,z]].$$
Therefore, $Y$ is smooth at $y$.
\end{proof}

\begin{lemma}\label{lem-connected}
Consider the commutative diagram
$$\xymatrix{
W_1\ar@{=}[d]^{p_1}& W_2\ar[l]_{\sigma_2}\ar[d]^{p_2} &\cdots\ar[l]_{\sigma_3} &W_{n-1}\ar[l]_{\sigma_{n-1}}\ar[d]^{p_{n-1}} &W_n\ar[l]_{\sigma_n}\ar[d]^{p_n}\\
X_1\ar[r]_{\pi_1}& X_2\ar@{-->}[r]_{\pi_2} &\cdots\ar@{-->}[r]_{\pi_{n-2}} &X_{n-1}\ar@{-->}[r]_{\pi_{n-1}} &X_n
}$$
where $X_1$ is a smooth projective threefold, $\pi_i$ is either a divisorial contraction or a flip, $W_i$ is the normalization of the graph of $\pi_{i-1}\circ p_{i-1}:W_{i-1}\dashrightarrow X_i$.
Let $x_n\in \Sing(X_n)$.
Then there exists a connected reduced divisor $E\subseteq \Exc(p_n)$ such that $p_n^{-1}(x_n)\subseteq E$ and $E_1$ is contracted to a point in some $X_i$ for some prime divisor $E_1\subseteq E$.
\end{lemma}

\begin{proof}
We show by induction on $n$.
Note that $\pi_1$ is a divisorial contraction by \cite[Theorem 3.3]{Mor82}.
When $n=2$, if $X_2$ is singular, we simply take $E=\Exc(\pi_1)$ by noting that $\Sing(X_2)=\pi_1(\Exc(\pi_1))$ is a point.
In general, if we find some connected reduced divisor $F\subseteq\Exc(p_{n-1})$, then $(p_{n-1}\circ\sigma_n)^{-1}(p_{n-1}(F))$ is connected since $p_{n-1}\circ\sigma_n$ has connected fibres.
Note that $\pi_{n-1}^{-1}$ contracts no divisors.
So $(p_{n-1}\circ\sigma_n)^{-1}(p_{n-1}(F))\subseteq \Exc(p_n)$.
Then we can take $E_F$ as the connected component of $\Exc(p_n)$ containing $(p_{n-1}\circ\sigma_n)^{-1}(p_{n-1}(F))$.
Note that $p_n^{-1}(p_n(E_F))$ is connected.
So $E_F= p_n^{-1}(p_n(E_F))$ and hence $p_n^{-1}(x_n)\subseteq E_F$ if and only if $x_n\in p_n(E_F)$.

\textbf{Case: $\pi_{n-1}$ is isomorphic at $x_n$.}

Let $x_{n-1}=\pi_{n-1}^{-1}(x_n)$.
Then $x_{n-1}\in\Sing(X_{n-1})$.
By induction, there exists a connected reduced divisor $F\subseteq \Exc(p_{n-1})$ such that $p_{n-1}^{-1}(x_{n-1})\subseteq F$ and $F_1$ is contracted to a point in some $X_i$ with $i\le n-1$ for some prime divisor $F_1\subseteq F$.
Then we can take $E=E_F$.

\textbf{Case: $\pi_{n-1}$ is not isomorphic at $x_n$ and $\pi_{n-1}$ is a divisorial contraction.}

Let $P$ be the exceptional divisor of $\pi_{n-1}$.
If $\pi_{n-1}(P)$ is a point, then $x_n=\pi_{n-1}(P)$ and we are done by simply taking $E$ as the connected component of $\Exc(p_n)$ containing $P$.
So we may assume $\pi_{n-1}(P)=C$ is a curve.
Note that $x_n\in C\cap\Sing(X_n)$.
By Lemma \ref{lem-mori}, we can find some $x_{n-1}\in \pi_{n-1}^{-1}(x_n)\cap \Sing(X_{n-1})$.
By the same argument as in the first case, we can find the desired $F$.
Note that $x_{n-1}\in P\cap p_{n-1}(F)$ and $(p_{n-1}\circ\sigma_n)^{-1}(P)\subseteq \Exc(p_n)$.
By the construction, $(p_{n-1}\circ\sigma_n)^{-1}(P)\subseteq E_F$ and $x_n\in p_n(E_F)$.
So we can take $E=E_F$.

\textbf{Case: $\pi_{n-1}$ is not isomorphic at $x_n$ and $\pi_{n-1}$ is a flip.}

Let $C\subseteq X_{n-1}$ and $C^+\subseteq X_n$ be the flipping loci which are connected by \cite[Corollary 1.3]{Mor88}.
Then there exists a $p_n$-exceptional connected reduced divisor $P$ of $W_n$ such that $p_{n-1}\circ\sigma_n(P)=C$ and $p_n(P)=C^+$.
By Lemma \ref{lem-ben}, there exists some $x_{n-1}\in C\cap \Sing(X_{n-1})$.
By the same argument as in the previous cases, we can find the desired $F$ with $P\subseteq E_F$.
Note that $\pi_{n-1}$ is not isomorphic at $x_n$.
Then $x_n\in C^+\subseteq p_n(E_F)$.
So we can take $E=E_F$.
\end{proof}

\begin{theorem}\label{thm-emmp-flip}
Consider the equivariant dynamical systems of a minimal model program
$$f_1\acts X_1\to X_2\dashrightarrow\cdots\dashrightarrow X_n\racts f_n$$ 
where $f_1$ is a non-isomorphic surjective endomorphism of a smooth projective threefold $X_1$.
Let $\pi_n:X_n\dashrightarrow X_{n+1}$ be a flip.
Then $\pi_n$ is $f_n$-equivariant after iteration of $f$.
\end{theorem}
\begin{proof}
Let $C$ be the flipping locus of $\pi_n$.
By Lemma \ref{lem-ben}, there exists some $x_n\in \Sing(X_n)\cap C$.
Let $\ell$ be an irreducible curve in $C$ containing $x_n$.
Denote by $\ell_m=(f_n)^m(\ell)$ when $m\ge 0$.
Note that $(f_n)^m(C)$ has the same number of irreducible components when $m\gg 1$.
Then $f_n^{-1}(\ell_{m+1})=\ell_m$ when $m\gg 1$ (cf.~\cite[Lemma 4.2]{MZ20a}).
Write $(f_n)^*\ell_{m+1}= a_m \ell_m$ with $a_m$ being a positive integer.
Suppose the contrary that the theorem is false, i.e., the $\ell_m$ are all different.

We apply Lemma \ref{lem-connected} and use the same notation there.
There exists a connected $p_n$-exceptional divisor $E$ such that $p_n^{-1}(x_n)\subseteq E$ 
and some prime divisor $E_1\subseteq E$ is contracted to a point in some $X_i$ with $i\le n$.
We may write $E=\sum\limits_{t=1}^s E_t$ such that the $E_t$ are prime divisors with $E_t\cap E_{t+1}\neq \emptyset$.
Note that $f_n$ lifts to $h_n:W_n\to W_n$.
After iteration, we may assume all $E_t$ are $h_n^{-1}$-invariant.
Note that $\dim p_n(\Exc(p_n))\le 1$ and $p_n(\Exc(p_n))$ is $f_n^{-1}$-invariant.
So $\ell_m\not\subseteq p_n(\Exc(p_n))$.
Let $\gamma_m$ be the strict transform of $\ell_m$ in $W_n$.
Note that $x_n\in \ell_0$ and $p_n^{-1}(x_n)\subseteq E$.
So $\gamma_0\cap E\neq\emptyset$ and hence  $\gamma_m\cap E\neq\emptyset$ for any $m\ge 0$.
By Lemma \ref{lem-eq-cartier}, $h_n|_{E_1}$ is $q$-polarized with $q=(\deg h)^{1/3}$.
By the Connecting Lemma \ref{lem-connection}, each $h_n|_{E_t}$ is $q$-amplified (by induction on $t \ge 1$).
In particular, we have 
$h_n^*E=qE$.

By Lemma \ref{lem-eq-cartier}, there is a $p_n$-exceptional effective Cartier divisor $D$ such that $E\subseteq \Supp D$ and $h_n^*D=qD$.
Since $\gamma_m\not\subseteq \Exc(p_n)$ and $\gamma_m\cap E\neq\emptyset$ for any $m\ge 0$, we have $D\cdot \gamma_m\in \mathbb{Z}_{>0}$ and
$$1\le D\cdot \gamma_{m+1}=\frac{1}{q^3}h_n^*D\cdot h_n^*\gamma_{m+1}=\frac{a_m}{q^2}D\cdot \gamma_m.$$
So $a_m\ge q^2>1$ for $m\gg 1$.
Note that $X_n$ has only finitely many singular points.
By the purity of branch locus, $\ell_m\subseteq \Supp R_{f_n}$ when $m\gg 1$.
After iteration, we may assume $P:=\overline{\bigcup_{m\gg 1} \ell_m}\subseteq \Supp R_{f_n}$ is a prime divisor of $X_n$.
Note that for any prime divisor $P'$ with $f_n(P')=P$, we have $P'\subseteq \overline{\bigcup_{m-1\gg 1} \ell_m}$.
So $f_n^{-1}(P)=P$.
Let $Q$ be the strict transform of $P$ in $W_n$.
Note that $\gamma_m\cap E\neq \emptyset$ implies $Q\cap E\neq \emptyset$.
By the Connection Lemma \ref{lem-connection}, $h_n^*Q=qQ$ and hence $f_n^*P=qP$.
Then we have 
$$(f_n|_P)^*\ell_{m+1}=\frac{a_m}{q} \ell_m,$$
where $\frac{a_m}{q}\ge q>1$ when $m\gg 1$.
So $f|_P$ is ramified along infinitely many curves, a contradiction as our base field has characteristic $0$.
\end{proof}

\section{Proof of Theorem \ref{mainthm-A}}

\begin{proof}[Proof of Theorem \ref{mainthm-A}] 
Suppose the contrary $f$ is $\delta$ primitive, and some birational MMP
$$\xymatrix{
X=X_1\ar[r]^{\pi_1}& X_2\ar@{-->}[r]^{\pi_2} &\cdots\ar@{-->}[r]^{\pi_{n-1}} &X_{n}\ar@{-->}[r]^{\pi_{n}} &X_{n+1}
}$$
has $\pi_i$ being $f_i:=f|_{X_i}$-equivariant after iteration for $i<n$ and $\pi_n$ not being $f_n$-equivariant even after iteration.
By Theorem \ref{thm-emmp-flip}, $\pi_n$ is a divisorial contraction.
By \cite[Lemma 6.2]{MZ18}, the exceptional divisor of $\pi_n$ is not $f_n^{-1}$-periodic.
By Theorem \ref{thm-E-Rf-notbig}, $\kappa_{f_n}(X,R_{f_n})<3$.
By Corollary \ref{cor-kappa-rf12}, $\kappa_{f_n}(X,R_{f_n})=0$ since ${f_n}$ is $\delta$-primitive.
Note that 
$$R_{f_n}=(\pi_{n-1}\circ\pi_{n-2}\circ\cdots\circ \pi_1)_*R_f.$$
So $\kappa_f(X, R_f)=0$ by Lemma \ref{lem-pushforward-fiitka}.
Hence the general fibre of $\phi$ is an elliptic curve, by Lemma \ref{lem-Rf=0}, 
Let $E_{X_n}$ be the exceptional divisor of $\pi_n$ and $E \subseteq X$ its strict transform.
By Lemma \ref{lem-fprime}, $E_{X_n}$ is $f_n$-prime, so $E$ is $f$-prime.
Let $E_i:=f^i(E)$ with $i\ge 0$.
Note that $E_i$ is covered by rational curves and $\kappa(X, E_i)=0$.
Let $\phi:X\dashrightarrow Y$ be the $f$-Iitaka fibration of $E$.
By Theorem \ref{thm-fiitaka2}, $\dim Y=2$, $\deg f|_Y=1$ and $\phi$ is almost well-defined.
By Corollary \ref{cor-kappa-rf0}, $f$ is quasi-\'etale and hence \'etale since $X$ is smooth.
In particular, $X$ admits a non-trivial \'etale cover.
So $X$ is not rationally connected (cf.~\cite[Corollary 4.18]{Deb01}).
If $X$ is non-uniruled, then we can always run smooth birational EMMP for non-uniruled $X$ (cf.~\cite[Section 4]{Fuj02}), a contradiction.
So, by \cite[Theorem 4.19]{Nak10}, there exists an $f$-equivariant special MRC fibration 
$$\psi:X\dashrightarrow Z$$
where $Z$ is non-uniruled normal projective variety  with $0<\dim Z<3$ and $\psi$ has rationally connected general fibres.
If $\phi \times \psi: X\dashrightarrow Y\times Z$ is generically finite, then $f$ is $\delta$-imprimitive by Lemma \ref{lem-imprimitive}, a contradiction.
Note that the induced map $X\dashrightarrow Y\times Z$ is generically finite when $\dim Z=2$.
So $\dim Z=1$ and $\psi$ factors through $\phi$.

Write $\psi=\sigma\circ \phi$ where $\sigma:Y\dashrightarrow Z$ is a dominant rational map.
Let $C$ be the image of $E_i$ in $Y$ with $i\gg 1$.
Then $C$ is a rational curve since the general fibre of $\phi$ is elliptic and $E_i$ is covered by rational curves.
Note that $Z$ is not a rational curve.
So $C$ and hence $E_i$ do not dominate $Z$.
Since $E_i$ is not $f^{-1}$-periodic, $Z$ is an elliptic curve and $\psi$ is a well-defined fibration.
So $E_i=\psi^{-1}(\psi(E_i))$ for $i\gg 1$ since $\psi$ has irreducible general fibre.
In particular, $E_i$ and hence $E=f^{-i}(E_i)$ are semi-ample, a contradiction.
\end{proof}

The following example is our first test on Theorem \ref{mainthm-A} before it is proved.
\begin{example}\label{example-emmp}
There exists a smooth rational surface $S$ admitting an automorphism $g$ of positive entropy with no $g$-periodic curves (cf.~\cite{Les21}).
There are infinitely many choices of MMPs from $S$, but none of them is $g$-equivariant even after iteration.
Let $E$ be an elliptic curve and $n_E:E\to E$ the multiplication map by $n$ with $n^2 > \delta_g$.
Let $X:=S\times E$ and $f=g\times n_E$.
We claim that $X$ admits no $f^{-1}$-periodic prime divisor.
Indeed, let $P$ be an $f^{-1}$-invariant prime divisor and $e\in E$ the identity element.
Then $P\cap S\times \{e\}$ is $(f|_{S\times\{e\}})^{-1}$-invariant and hence $P=S\times \{a\}$ for some $a$. 
However, $E$ admits no $(n_E)^{-1}$-periodic point.
So the claim is true.
Hence there are infinitely many choices of MMPs from $X$, but none of them is $f$-equivariant even after iteration.
Nevertheless, $f$ is still $\delta$-imprimitive, via the projection $X\to E$.
\end{example}

\section{Fano contraction to a curve}

In this section, we focus on the case of a Fano contraction $X \to Y$ to a curve (hence $\rho(X) = 2$). Theorem \ref{thm-fano31} is our main result.
By a Fano contraction of $X$, we always mean a Fano contraction of a $K_X$-negative extremal ray. We begin with some lemmas.

\begin{lemma}\label{lem-2fano}
Let $f:X\to X$ be a $\delta$-primitive endomorphism of a normal projective variety $X$.
Then $X$ admits at most one Fano contraction.
\end{lemma}
\begin{proof}
Suppose the contrary that there are two different Fano contractions $\pi:X\to Y$ and $\tau:X\to Z$.
After iteration, we may assume $\pi$ and $\tau$ are $f$-equivariant (cf.~\cite[Remark 6.3]{MZ18}).
Since $\rho(X)=\rho(Y)+1=\rho(Z)+1$, we have $\N^1(X)=\pi^*\N^1(Y)+\tau^*\N^1(Z)$.
Therefore, $\delta_f=\max\{\delta_{f|_Y},\delta_{f|_Z}\}$.
Note that $Y$ and $Z$ must be positive dimensional.
So $f$ is $\delta$-imprimitive.
\end{proof}

\begin{lemma}\label{lem-degg>1}
Consider the following equivariant dynamical systems of smooth projective varieties
$$\xymatrix{
f \acts X\ar[r]^{\tau} &Y\racts g
}$$
where $\dim X=3$, $\dim Y=1$, and $\tau$ is a Fano contraction. 
Suppose $\deg g>1$.
Then $f$ is $\delta$-imprimitive.
\end{lemma}
\begin{proof}
Suppose the contrary $f$ is $\delta$-primitive.
Then $\delta_f>\delta_g>1$ and $f$ is int-amplified (cf.~\cite[Theorem 1.1]{Men20}).
Then $-K_X$ is $\Q$-effective (cf.~\cite[Theorem 1.5]{MZ22}).
We apply the proof of \cite[Proposition 9.2]{MZ22}.
First, we can find some nef $\tau$-ample Cartier divisor $D$ such that $f^*D\equiv \delta_fD$, see the second paragraph of the proof of \cite[Proposition 9.2]{MZ22}.
Let $a>0$ be a rational number such that $B:=D+aK_X$ is $\tau$-trivial.

In the Situation 1 of the proof of \cite[Proposition 9.2]{MZ22}, i.e., 
when $B$ is pseudo-effective, we have $\kappa(X,-K_X)=0$ and $-K_X\sim_{\Q} D$ where $D=\Supp R_f$ is an $f^{-1}$-invariant prime divisor.
This is also Case TIR (cf.~\cite[Section 6]{MMSZ23}) and it does not occur by \cite[Theorem 6.6]{MMSZ23}.

So we are left with Situation 2 in the proof of \cite[Proposition 9.2]{MZ22}: $B$ is not pseudo-effective.
Note that $\rho(X)=2$, $B$ is not nef, and $B$ is $\tau$-trivial.
Then the extremal ray of $\NE(X)$ not contracted by $\tau$ is $B$-negative and hence $K_X$-negative since $D$ is nef.
Let $\phi:X\to Z$ be the $f$-equivariant contraction of the $B$-negative extremal ray.
If $\dim Z=\dim X$, then $f|_Z$ and hence $f$ and $g$ are $\delta_f$-polarized (cf.~\cite[Section 3]{MZ18}), a contradiction.
If $\dim Z<\dim X$, then $f$ is $\delta$-imprimitive by Lemma \ref{lem-2fano}.
\end{proof}

The following two lemmas collect the common part that we shall use in the proof of Lemmas \ref{lem-elliptic} and \ref{lem-rational}.

\begin{lemma}\label{lem-2comp}
Let $\pi:D\to Y$ be a surjective morphism where $D$ is a projective surface and $Y$ is a smooth projective curve.
Suppose the general fibre of $\pi$ is a simple normal crossing loop of $m$ smooth rational curves.
Let $Z$ be the union of curves in $\Sing(D)$ dominating $Y$.
Then $Z$ has at most $2$ irreducible components.
\end{lemma}

\begin{proof}
Let $\nu:\overline{D}\to D$ be the normalization.
Let $y\in Y$ be a general point.
Then the general fibre $\overline{D}_y$ of $\pi\circ \nu$ has $m$ connected (and irreducible) components.
Let $Q$ be an irreducible component of $\overline{D}_y:=(\pi\circ \nu)^{-1}(y)$.
Note that $\nu|_{Q}$ is isomorphic.
So we have
$$\nu^{-1}(Z)\cap \overline{D}_y=\bigsqcup_{Q\subseteq \overline{D}_y} (\nu^{-1}(Z)\cap Q)=\bigsqcup_{Q\subseteq \overline{D}_y} (\nu^{-1}(Z\cap \nu(Q))\cap Q).$$
Note that $\sharp Z\cap \nu(Q)=2$.
Then $\sharp \nu^{-1}(Z)\cap \overline{D}_y= 2m$ and hence $\deg (\pi|_D\circ \nu)|_{\nu^{-1}(Z)}=2m$.
Let $p_1:\overline{D}\to \overline{Y}$ and $p_2:\overline{Y}\to Y$ form the Stein factorization of $\pi\circ \nu$.
Note that $\deg p_2=m$.
Then $\deg (p_1)|_{\nu^{-1}(Z)}=2m/m=2$.
Since $\overline{Y}$ is irreducible, $\nu^{-1}(Z)$ and hence $Z$ have at most two irreducible components.
\end{proof}

\begin{lemma}\label{lem-fano-toric}
Consider the following equivariant dynamical systems of smooth projective varieties
$$\xymatrix{
f \acts X\ar[r]^{\tau} &Y\racts g
}$$
where $\dim X=3$, $\dim Y=1$, and $\tau$ is a Fano contraction. 
Suppose $f$ is $\delta$-primitive and non-isomorphic.
Then $\deg g=1$ and the following hold (after iterating $f$).
\begin{enumerate}
\item $B_f=\Supp R_f$ is an $f^{-1}$-invariant, nef and $\tau$-ample prime divisor 
such that $\kappa(X, B_f)=0$,  $(X, B_f)$ is lc and $f^*B_f = \delta_f B_f$.
\item Let $Z$ be the union of $f^{-1}$-periodic curves dominating $Y$. Then $Z$ is the union of curves $C\subseteq \Sing(B_f)$ with $C$ dominating $Y$.
\item $B_f\cap X_y$ is a simple normal crossing loop of $m$ smooth rational curves with $m\ge 3$ for general $y\in Y$.
In particular, $\deg \tau|_Z=m\ge 3$.
\item $Z$ has at most $2$ irreducible components.
\end{enumerate}
\end{lemma}

\begin{proof}
Since $f$ is $\delta$-primitive, $\delta_f>\delta_g$.
By Lemma \ref{lem-degg>1}, $\deg g=1$.
Let $V$ be the $1$-dimensional $\delta_f$-eigenspace of $f^*|_{\N^1(X)}$ (generated by a nef divisor).
The other eigenvalue of $f^*|_{\N^1(X)}$ is $\delta_g=1$.
Then $V=\text{Im} (f^*-\id)|_{\N^1(X)}$.
By the ramification divisor formula, 
$$R_f=f^*(-K_X)-(-K_X)\in V$$ and $f^*R_f\equiv \delta_f R_f$.
Note that $f$ is not polarized (or else, $\delta_f = \delta_g$, absurd).
So $R_f$ is not big even after any iteration (cf.~\cite[Proposition 1.1]{MZ18}).
Then $\kappa_f(X, R_f)<3$.

Let $y, y'\in Y$ be general points such that $g(y)=y'$.
We may assume that the fibres $X_{y'}$ and $X_y$ are smooth del pezzo surfaces with the same Picard number.
Note that $\deg f|_{X_y}=\deg f>1$.

(1) 
By \cite[Corollary 4.18]{Deb01}, $X_{y'}$ has no non-trivial \'etale cover.
By the purity of branch locus, $R_{f|_{X_y}}\neq 0$.
By the adjunction formula and the ramification divisor formula, $R_f|_{X_y}=R_{f|_{X_y}}\neq 0$.
So $R_f$ is $\tau$-ample and also an extremal ray of $\PE^1(X) = \Nef(X)$.

Suppose the contrary that $\kappa_f(X, R_f)>0$.
Then $\kappa(X, R_f)>0$ after iteration.
Let $\phi:X\dashrightarrow Z$ be the $f$-Iitaka fibration of $R_f$.
By Theorem \ref{thm-fiitaka-polarized}, $f|_Z$ is $\delta_f$-polarized.
In particular, $f$ is $\delta$-imprimitive, a contradiction.
So $\kappa_f(X, R_f)=\kappa(X,B_f)=0$ (cf.~Proposition \ref{prop-f*f*}).
By Lemma \ref{lem-fkappa>0}, each component of $B_f=\Supp R_f$ is $f^{-1}$-invariant after iteration.
By \cite[Proposition 2.1]{Zha14}, the pair $(X,B_f)$ is log canonical.

Suppose the contrary that $B_f$ contains at least two irreducible components $P_1$ and $P_2$.
Since $R_f$ is extremal in $\PE^1(X)$, $P_1\equiv tP_2$ for some rational number $t>0$.
Then $P_1-tP_2=\tau^*Q$ for some $\Q$-Cartier divisor $Q$ on $Y$ by the cone theorem (cf. \cite[Theorem 3.7]{KM98}).
Note that 
$g^*Q=\delta_f Q$ and $\delta_f>\delta_g$.
So $Q\equiv 0$ and hence $rQ\in \Pic^0(Y)$ for some integer $r>0$.
Note that $\delta_f^2>\delta_g$.
By \cite[Proposition 6.3]{MMS+22}, $rQ\sim_{\Q} 0$  and hence $P_1\sim_{\Q} tP_2$.
Then $\kappa(X,R_f)\ge \kappa(X, P_1)> 0$, a contradiction.
So $B_f$ is irreducible.

(2) 
By choosing $y$ general, we may assume 
$$(\Supp R_f)|_{X_y}=B_f|_{X_y}=B_f\cap X_y=\Supp R_{f|_{X_y}}.$$ 
Note that 
$$K_X+B_f=f^*(K_X+B_f)$$ and $\delta_f>1$.
So $K_X+B_f\in \tau^*\N^1(Y)$ is $\tau$-trivial.
By the adjunction formula, 
$$K_{X_y}+B_f\cap X_y=(K_X+B_f)|_{X_y}=(f|_{X_y})^*(K_{X_{y'}}+B_f\cap X_{y'})\equiv 0.$$
Note that $B_f\cap X_y$ has the same number of irreducible components for general $y$.
So $f|_{X_y}$ is totally ramified.
By \cite[Theorem 5.50]{KM98}, $(X_y, R_{f|_{X_y}})$ is log Calabi-Yau. 
By Lemma \ref{lem-surf-toric}, $(X_y, R_{f|_{X_y}})$ is further a toric pair and the unsplitting points of $f|_{X_y}$ are $\Sing(B_f\cap X_{y'})=\Sing(B_f)\cap X_{y'}$.
Let $C$ be an $f^{-1}$-periodic curve dominating $Y$.
Choosing $y, y'$ general, we may assume $\sharp C\cap X_y=\sharp C\cap X_{y'}$.
Then the points in $C\cap X_{y'}$ are $f|_{X_y}$-unsplitting and hence $C\cap X_{y'}\subseteq \Sing(B_f)\cap X_{y'}$.
So $C\subseteq \Sing(B_f)$.

On the other hand, each irreducible curve $C'$ (not necessarily dominating $Y$) of $\Sing(B_f)$ is an lc centre of $(X,B_f)$ (cf.~\cite[Lemma 2.29]{KM98}).
So $f^{-1}(C')=C'$ after iteration by \cite[Lemma 2.10]{BH14}.
Therefore, $Z$ is the union of of curves $C\subseteq \Sing(B_f)$ with $C$ dominating $Y$.

(3) Let $y\in Y$ be a general point.
By the classification of smooth toric Fano surfaces, the boundary $B_f\cap X_y$ is a simple normal crossing loop of $m$ smooth rational curves with $m\ge 3$.
So $\sharp Z\cap X_y=\sharp \Sing(B_f\cap X_y)=m\ge 3$.
Therefore, $\deg \tau|_Z\ge 3$.

(4) This follows from Lemma \ref{lem-2comp}.
\end{proof}


\begin{lemma}\label{lem-elliptic}
Consider the following equivariant dynamical systems of smooth projective varieties
$$\xymatrix{
f \acts X\ar[r]^{\tau} &Y\racts g
}$$
where $\dim X=3$, $Y$ is an elliptic curve, $\deg g=1$, and $\tau$ is a Fano contraction. 
Suppose $f$ is $\delta$-primitive and non-isomorphic.
Then $g=\id$ after iteration.
\end{lemma}
\begin{proof}
Suppose the contrary that $g$ is not an automorphism of finite order.
Since $Y$ is an elliptic curve, $g$ is a non-torsion translation.
Therefore,
$$\textbf{``any $f^{-1}$-periodic irreducible closed subvariety of $X$ will dominate $Y$''}$$ by \cite[Lemma 7.5]{CMZ20}.
In particular, any $f^{-1}$-periodic irreducible curve is an elliptic curve \'etale over $Y$.
Let $D:=B_f$ and $Z$ the union of $f^{-1}$-periodic irreducible curves.
By Lemma \ref{lem-fano-toric}, $D$ is a nef $\tau$-ample prime divisor, $Z$ is the union of curves dominating $Y$ in $B_f$, and $\deg \tau|_Z\ge 3$.
In particular, $D$ contains an $f^{-1}$-periodic irreducible curve.
Note that $f^*D=\delta_f D$ and $D^2\cdot X_y>0$ for any general fibre $X_y$ of $\tau$.
By the projection formula, $\deg f=\delta_f^2$.

{\it We do replacements of $\tau:X\to Y$ to make every $f^{-1}$-periodic curve a section over $Y$.}

Let $C$ be an $f^{-1}$-periodic curve such that $\deg \tau|_C>1$.
Then the natural embedding $C\hookrightarrow C\times_Y C$ implies that the preimage of $C$ in $X\times_Y C$ splits into at least two irreducible components.
Each time, we replace $\tau:X\to Y$ by $X\times_Y C\to C$, $D$ by $D\times_Y C$, and $f,g$ by the natural liftings.
So after a suitable \'etale base change, we may assume $\deg \tau|_C=1$ for every $f^{-1}$-periodic curve $C$. 
However, $\tau$ may no longer be a Fano contraction and $D$ may no longer be irreducible.
Note that $g$ is still a non-torsion translation.
So any $f^{-1}$-periodic irreducible closed subvariety of $X$ still dominates $Y$.
Then $D$ is the union $f^{-1}$-periodic irreducible closed subvarieties with dimension $\le 2$.
Moreover, each irreducible component of $D$ contains an $f^{-1}$-periodic irreducible curve.

$K_X$ is still not pseudo-effective since the general fibres over $Y$ are still toric surfaces.
By Theorem \ref{mainthm-A}, we can further run an $f$-equivariant (after iteration) birational MMP
$$X\dashrightarrow X_2\dashrightarrow\cdots \dashrightarrow X_r$$ and a Fano contraction $\tau_r:X_r\to Y'$.
Since $Y$ is an elliptic curve, the MMP is over $Y$.
Note that $X$ is still smooth and there is no $g^{-1}$-periodic point on $Y$.
So each $X_i\to X_{i+1}$ has to be a divisorial contraction with the exceptional divisor contracted to an elliptic curve \'etale over $Y$ (cf.~\cite[Lemma 7.5]{CMZ20}).
In particular, $X_r$ is still smooth.
Since $f$ is $\delta$-primitive, $\delta_{f|_{Y'}}<\delta_f$.
Note that $\deg f=\delta_f^2$.
If $\dim Y'=2$, then $\delta_f = \deg f|_{X_{y'}} = (\deg f)/(\deg f|_{Y'})$,
$\deg f|_{Y'}=\delta_f$ and $\iota_{f|_{Y'}}=\delta_f/\delta_{f|_{Y'}}>1$ by Lemma \ref{lem-surf-deg}.
In particular, $f|_{Y'}$ and hence $g$ are int-amplified, a contradiction.
So $Y'\to Y$ is isomorphic by noting that $Y'\to Y$ has connected fibres.
Let $C_r$ be an $f_r^{-1}$-periodic irreducible curve of $X_r$.
Let $V$ be the preimage of $C_r$ in $X$.
Since $V$ is $f^{-1}$-periodic, $V\subseteq D$.
Now either $V$ is a curve or $V$ is an irreducible component of $D$ containing an $f^{-1}$-periodic curve.
So $C_r$ is always dominated by some $f^{-1}$-periodic curve in $X$.
In particular, $\deg \tau_r|_{C_r}=1$.
Now, replacing $\tau$ by $\tau_r$, we finally have that $\tau$ is a Fano contraction as in Lemma \ref{lem-fano-toric}; meanwhile, $\deg \tau|_C=1$ holds for each $f^{-1}$-periodic curve $C$.

We now apply Lemma \ref{lem-fano-toric} to our new setting again.
So $\deg \tau|_Z\ge 3$ and $Z$ has at most $2$ irreducible components.
However, this is impossible since each irreducible component of $Z$ is only a section over $Y$.
\end{proof}

\begin{lemma}\label{lem-rational}
Consider the following equivariant dynamical systems of smooth projective varieties
$$\xymatrix{
f \acts X\ar[r]^{\tau} &Y\racts g
}$$
where $\dim X=3$, $Y$ is a rational curve, $\deg g=1$, and $\tau$ is a Fano contraction. 
Suppose $f$ is $\delta$-primitive and non-isomorphic.
Then $g=\id$ after iteration.
\end{lemma}

\begin{proof}
Suppose the contrary that $g$ is an automorphism of infinite order.
Then $g$ has at most $2$ periodic points.
By Lemma \ref{lem-fano-toric}, $D:=B_f\neq 0$ is an $f^{-1}$-invariant prime divisor and $\kappa(X, D)=0$.
Let $Z$ be the union of $f^{-1}$-periodic curves dominating $Y$.
Then $\deg \tau|_Z\ge 3$ and $Z$ has at most $2$ irreducible components.
So there is an irreducible component $C$ of $Z$ such that $q:=\deg \tau|_C>1$.
Let $\widetilde{C}$ be the normalization of $C$ and $p:\widetilde{C}\to Y$ the induced finite surjective morphism with $\deg p=q>1$.
Let $\widetilde{X}$ the normalization of $X\times_Y \widetilde{C}$.
Note that we can not guarantee the smoothness or even $\Q$-factorial singularities for $\widetilde{X}$.

Now we have the following equivariant dynamical system:
$$\xymatrix{
\save[]+<2pc,0.22pc>*{\widetilde{f} \acts} \restore &\widetilde{X} \ar[r]^{\widetilde{p}}\ar[d]_{\widetilde{\tau}}  & X \ar[d]^{\tau}  &\save[]+<-2pc,0.1pc>*{\racts f} \restore\\
\save[]+<2pc,0.1pc>*{\widetilde{g} \acts} \restore&\widetilde{C}\ar[r]_p& Y&\save[]+<-2pc,0.05pc>*{\racts g} \restore
}$$
where $\widetilde{g}$ is an automorphism of infinite order and hence $\widetilde{C}$ has genus $\le 1$.
Note that $g$ has at least one fixed point by the Lefschetz trace formula.
So $\widetilde{g}$ has a periodic point.
Recall that an automorphism of an elliptic curve has finite order if it has a periodic point.
Therefore, we have $\widetilde{C}\cong \mathbb{P}^1$.
Since $\mathbb{A}^1$ and $\mathbb{P}^1$ are simply connected, the branch divisor $B_p$ contains at least $2$ points.
On the other hand, $B_p$ is $g$-invariant and $g$ has at most $2$ periodic points.
So after choosing a suitable coordinate of $Y$, we may assume $B_p=y_0+y_{\infty}$ where $y_0=[0:1]$ and $y_{\infty}=[1:0]$.
After iteration, we may write $g([x_0:x_1])=[ax_0:x_1]$ for some non-zero $a$ which is not a root of unity.
By choosing a suitable coordinate of $\widetilde{C}\cong Y$, we may write $p([x_0:x_1])=[x_0^q:x_1^q]$ and $\widetilde{g}([x_0:x_1])=[a^{1/q}x_0:x_1]$ satisfying $g\circ p=p\circ\widetilde{g}$.
Note that 
$$\widetilde{C}\times_Y \widetilde{C}=\{([a_0:a_1],[b_0:b_1])\in \widetilde{C}\times \widetilde{C}\,|\, [a_0^q:a_1^q]=[b_0^q:b_1^q]\}$$ and  $a_0^qb_1^q-b_0^qa_1^q=\prod_{i=1}^q(a_0b_1-\xi_q^ib_0a_1)$ where $\xi_q$ is the primitive $q$-th root of unity.
So $\widetilde{C}\times_Y \widetilde{C}$ has exactly $q$ irreducible components.
Recall that $\deg \tau|_Z\ge 3$.
Then $\widetilde{p}^{-1}(Z)$ has at least $3$ irreducible components.

Let $y\in Y$ and $P$ an irreducible component of $\tau^{-1}(y)$.
Since $X$ is smooth, P is Cartier.
By the cone theorem (cf.~\cite[Theorem 3.7]{KM98}), $P=\tau^*y$.
So each fibre of $\tau$ is irreducible and reduced.
Note that $\widetilde{p}^*X_{y_0}=\widetilde{\tau}^*p^*y_0=q\widetilde{\tau}^*y_0$.
Then $\widetilde{\tau}^*y_0$ is irreducible and reduced by noting that $\deg \widetilde{p}=q$.
In particular, $\deg \widetilde{p}|_{\widetilde{X}_{y_0}}=1$.
The same argument holds for $y_{\infty}$.
Note that $\widetilde{p}$ is \'etale outside $X_{y_0}$ and $X_{y_\infty}$.
So $\widetilde{\tau}$ has irreducible and reduced fibres.

Let $V$ be an irreducible component of the non-lc locus of the pair $(X, D+X_{y_0}+X_{y_\infty})$.
By Lemma \ref{lem-fano-toric}, $(X,D)$ is log canonical.
Then $V\subseteq X_{y_0}\cup X_{y_\infty}$ and $\dim V\le 1$.
By \cite[Theorem 1.2]{BH14}, after iteration, $f^{-1}(V)=V$ and $\deg f|_V=\deg f\ge \deg f|_{X_{y_0}}$.
Note that $D|_{X_{y_0}}$ is ample and $(f|_{X_{y_0}})^*D|_{X_{y_0}}=\delta_f D|_{X_{y_0}}$.
So $f|_{X_{y_0}}$ and hence $f|_V$ are $\delta_f$-polarized.
Then $\deg f|_{X_{y_0}}>\deg f|_V$, a contradiction.
Therefore, $(X, D+X_{y_0}+X_{y_\infty})$ is log canonical.
In particular, $X_{y_0}$ is smooth at every ($1$-dimensional) generic point $\eta$ of $D\cap X_{y_0}$ by \cite[Lemma 2.29]{KM98}.
Then $\widetilde{p}|_{\widetilde{X}_{y_0}}$ is isomorphic over $\eta$ by the Zariski main theorem.
In particular, $\widetilde{p}^{-1}(x)$ consists of only $1$ point when $x$ is a general point of $D\cap X_{y_0}$, and we shall derive a contradiction based on this. 

Let $\widetilde{D}:=\widetilde{p}^*D=\widetilde{p}^{-1}(D)$.
Note that $K_{\widetilde{X}}+\widetilde{D}+\widetilde{X}_{y_0}+\widetilde{X}_{y_\infty}=\widetilde{p}^*(K_X+D+X_{y_0}+X_{y_\infty})$.
So the pair $(\widetilde{X}, \widetilde{D}+\widetilde{X}_{y_0}+\widetilde{X}_{y_\infty})$ is log canonical by \cite[Proposition 5.20]{KM98}.
Note that $\widetilde{p}^{-1}(Z)\subseteq \Sing(\widetilde{D})$ since $\widetilde{p}$ is \'etale outside $\widetilde{X}_{y_0}\cup \widetilde{X}_{y_\infty}$.
Applying Lemma \ref{lem-2comp} to $\widetilde{D}\to \widetilde{C}$, we have $\widetilde{D}$ reducible since $\widetilde{p}^{-1}(Z)$ has at least $3$ irreducible components.
Write $\widetilde{D}=\sum\limits_{i=1}^n \widetilde{D}_i$ where $\widetilde{D}_i$ is irreducible and $n\ge 2$.
Let $\ell_i:=\widetilde{D}_i\cap \widetilde{X}_{y_0}$.
Note that $\widetilde{p}(\ell_i)=D\cap X_{y_0}$. 
Then $\widetilde{X}$ is smooth at every generic point of $\ell_i$ (cf.~\cite[Lemma 2.51]{KM98}).
By \cite[Lemma 2.29]{KM98} again, $\ell_i$ and $\ell_j$ have no common irreducible component for $i\neq j$.
Let $x\in D\cap X_{y_0}$ be a general point.
Then $\widetilde{p}^{-1}(x)$ consists of at least $2$ points, a contradiction.
\end{proof}

\begin{theorem}\label{thm-fano31}
Consider the following equivariant dynamical systems of smooth projective varieties
$$\xymatrix{
f \acts X\ar[r]^{\tau} &Y\racts g
}$$
where $\dim X=3$, $\dim Y=1$, and $\tau$ is a Fano contraction of some $K_X$-negative extremal ray. 
Suppose $f$ is $\delta$-primitive and non-isomorphic.
Then $g=\id$ after iteration.
\end{theorem}
\begin{proof}
Since $f$ is $\delta$-primitive, $\deg g=1$ by Lemma \ref{lem-degg>1}.
Suppose $g$ is not an automorphism of finite order.
Then $Y$ is either rational or elliptic.
The theorem follows from Lemmas \ref{lem-elliptic} and \ref{lem-rational}.
\end{proof}

\begin{remark}
The proofs of Lemmas \ref{lem-elliptic} and \ref{lem-rational} adopt quite different strategies.
The proof of Lemma \ref{claim-pxp} is similar to the proof of \cite[Theorem 6.6]{MMSZ23}, which essentially makes use of ``$f^{-1}$-periodic subvarieties must dominate $Y$'' to keep the smoothness (or at least $\Q$-factorial singularities) of $X$ after base change and EMMP.
This phenomenon no longer holds in Lemma \ref{lem-rational} and we make a big effort to control the singularities of $X$ after base change.
\end{remark}

\section{Proof of Theorem \ref{mainthm-B}}

\begin{theorem}\label{thm-kappa<0}
Let $f:X\to X$ be a $\delta$-primitive non-isomorphic surjective endomorphism of a smooth projective threefold $X$ with Kodaira dimension $\kappa(X)<0$.
Then, after iteration, $f$ is either polarized or strongly imprimitive.
\end{theorem}
\begin{proof}
Note that $K_X$ is not pseudo-effective.
So we can run the following MMP
$$\xymatrix{
X_1\ar[r]^{\pi_1}& X_2\ar@{-->}[r]^{\pi_2} &\cdots\ar@{-->}[r]^{\pi_{n-2}} &X_{n-1}\ar@{-->}[r]^{\pi_{n-1}} &X_n\ar[r]^\tau&Y
}$$
where $\pi_i$ is either a divisorial contraction or a flip and $\tau$ is a Fano contraction.
By the termination of flips (cf.~\cite{Sho96}), we may assume that $X_n$ admits no $K_X$-negative extremal birational contraction.
Since $f$ is $\delta$-primitive, the MMP is $f$-equivariant after iteration by Theorem \ref{mainthm-A}.
Denote by $f_i:=f|_{X_i}$ and $g:=f|_Y$.
If $\dim Y=0$, then $f_n$ and hence $f$ are polarized (cf.~\cite[Corollary 3.12]{MZ18}).
So we may assume $0<\dim Y<3$.
Since $f$ is $\delta$-primitive, $\delta_f=\delta_{f_i}>\delta_g$.
By Lemma \ref{lem-2fano}, $X_n$ has only one Fano contraction which is $\tau$.
Then there is only one $K_{X_n}$-negative extremal ray.
If $\dim Y=2$, then $f$ is $\delta$-imprimitive by Corollary \ref{cor-fano-n-1} since $\pi_i^{-1}$ does not contract divisors.
This is a contradiction.
So $\dim Y=1$ and the Picard number $\rho(X_n)=2$.
Write $f_n^*D_+\equiv \delta_f D_+$ for some nef $\tau$-ample $D_+\in \N^1(X_n)$ and $f_n^*D_-\equiv \delta_g D_-$ for some $D_-\in \tau^*(\Amp(Y))$.
Then $D_+^2\cdot D_->0$ and hence $\deg f_n=\delta_f^2\cdot\delta_g$ by the projection formula.

We claim that $X_n$ is smooth.
Note that $\pi_1$ is a divisorial contraction by \cite[Theorem 3.3]{Mor82}.
Let $E$ be the exceptional divisor of $\pi_1$.
Write $f^*E=aE$.
If $\pi_1(E)$ is a point, then $E^3<0$ (cf.~\cite[Lemma 2.62]{KM98}) and $\deg f=a^3=\deg f_n=\delta_f^2\cdot \delta_g$.
In particular, $\delta_g<a<\delta_f$.
Moreover, $f|_E$ is $a$-polarized since $-E|_E$ is ample.
Let $W$ be the normalization of the graph of $X\dashrightarrow X_n$ and $h:W\to W$ the lifting of $f$.
Let $p_X:W\to X$ and $p_Y:W\to Y$ be the projections.
Let $E'$ be the strict transform of $E$ in $W$.
Then $h|_{E'}$ is $a$-polarized.
If $p_Y(E')=Y$, then $g$ is also $a$-polarized (cf.~\cite[Theorem 1.3]{MZ18}) and hence $a=\delta_g$, absurd.
So $y:=p_Y(E')$ is a point and $g^{-1}(y)=y$ by \cite[Lemma 7.5]{CMZ20}.
Let $D:=p_Y^*y$.
Then $E'\subseteq \Supp D$.
Note that $g^*y=\delta_g y$.
Then $h^*D=\delta_g D$, a contradiction because $h^*E'=aE'$.
So $\pi_1(E)$ is a curve.
By \cite[Theorem 3.3]{Mor82}, $X_2$ is smooth and $\pi_2$ is also a divisorial contraction.
Repeating the above argument for the subsequent steps, the claim is proved.

The theorem then follows from Theorem \ref{thm-fano31}.
\end{proof}

\begin{proof}[Proof of Theorem \ref{mainthm-B}]
(1) is just Theorem \ref{thm-kappa<0} and (3) follows from \cite[Theorem A]{NZ09}.

For (2), by \cite[Main Theorem (A), Case 3, Page 61]{Fuj02}, an \'etale cover $\widetilde{X}$ of $X$ is either an abelian veriety or $\widetilde{X}\cong S\times C$ where $S$ is a K3 surface and $C$ is an elliptic curve.
For the first case, $f$ is quasi-abelian.
Indeed, we can choose the cover suitably, called the Albanese closure (cf.~\cite[Lemmas 2.12]{NZ10}), 
such that $f$ lifts to a surjective endomorphism $\widetilde{f}:\widetilde{X}\to \widetilde{X}$; further, in the second case, 
$\widetilde{f}=g\times h$ splits as an automorphism $g:S\to S$ and a non-isomorphic surjective endomorphism $h:C\to C$ (cf.~\cite[Proposition 3.5]{NZ10} and \cite[Proposition 7.4]{Men23}).
By Lemma \ref{lem-imprimitive}, $\widetilde{f}$ is $\delta$-imprimitive.
\end{proof}

\section{Applications: Proofs of Theorems \ref{mainthm-C}, \ref{thm-sand} and \ref{mainthm-D}}

As applications of the EMMP Theorem \ref{mainthm-A} and the structure Theorem \ref{mainthm-B}, we prove Theorem \ref{mainthm-C} for the Kawaguchi-Silverman conjecture, Theorem \ref{thm-sand} for the sAND conjecture, and Theorem \ref{mainthm-D} for the Zariski dense orbit conjecture.

\begin{proof}[Proof of Theorem \ref{mainthm-C}]
Note that KSC holds for projective curves and surfaces and we may assume $f$ is weakly primitive and not quasi-abelian (cf.~Remark \ref{rmk-ksc}).
By Lemma \ref{lem-ksc-iff}, we may assume $f$ is strongly $\delta$-primitive.
By Theorem \ref{mainthm-B}, $f$ is polarized after iteration.
Note that KSC holds for polarized endomorphisms (cf.~Remark \ref{rmk-ksc}).
\end{proof}

We generalize the KSC Conjecture \ref{conj-KSC} to sAND Conjecture \ref{conj-sand} (cf.~\cite[Conjecture 1.3]{MMSZ23}).

\begin{definition}\label{def-sand}
Let $X$ be a projective variety and $f$ a surjective endomorphism on $X$ over a number field $K$ with a fixed algebraic closure $\overline{K}=\overline{\mathbb{Q}}$.
Let $d > 0$. 
Set
$$X(d)=X(K,d)=\{ x \in X(L) \mid K \subseteq L \subseteq \overline K,\ [L:K] \leq d\},$$
$$Z_{f}=Z_f( \overline{K})=\{ x \in X(\overline{K}) \mid  \alpha_f(x) < \delta_f \},$$
$$Z_{f}(d)=Z_f(K, d)=Z_f( \overline{K})\cap X(K,d).$$
\end{definition}

\begin{conjecture}[{\bf small Arithmetic Non-Density = sAND}]\label{conj-sand} 
Let $f:X\to X$ be a surjective endomorphism of a projective variety $X$ defined over a number field $K$.
Then the set $Z_f(K, d)$ is not Zariski dense in $X_{\overline{K}} = X \times_K \overline{K}$
for any positive constant $d>0$.
\end{conjecture}

\begin{theorem}\label{thm-sand}
Let $f:X\to X$ be a weakly primitive non-isomorphic surjective endomorphism of a smooth projective threefold $X$.
Then Conjecture sAND holds for $f$.
\end{theorem}
\begin{proof}
Note that Conjecture sAND holds for projective curves and surfaces by \cite[Theorem 4.4]{MMSZ23}.
Moreover, Conjecture sAND holds for quasi-abelian $f$ by \cite[Corollary 5.9]{MMSZ23}.
Then the proof of Theorem \ref{mainthm-C} works after replacing Lemma \ref{lem-ksc-iff} by \cite[Lemma 2.9]{MMSZ23}.
\end{proof}

\begin{remark}
We cannot fully deal with the strongly imprimitive $f$ for Conjecture sAND, which is also highly related to another very hard ``Uniform Boundedness Conjecture'' (cf.~\cite[Conjecture 1.9]{MMSZ23}), see \cite[Section 7]{MMSZ23} for some special cases.
\end{remark}

In the rest of this section, we treat the following conjecture and prove Theorem \ref{mainthm-D}.

\begin{conjecture}[{\bf Zariski Dense Orbit = ZDO}]\label{conj-zdo} 
Let $f: X \to  X$ be a surjective endomorphism of a projective variety. 
Then either $f$ is strongly imprimitive or $f$ has a Zariski dense orbit, i.e., there is a closed point $x\in X$ such that the orbit
$$O_f(x):=\{f^i(x)\,|\, i\ge 0\}$$
is Zariski dense in $X$. 
\end{conjecture}

Under a generically finite dominant map, the property of having a Zariski dense orbit is stable.
The following lemma says that the same is true for the weakly primitive property, see \cite[Lemma 2.1]{Xie22}.
\begin{lemma}\label{lem-zdo-gf}
Consider the following equivariant dynamical systems of projective varieties
$$\xymatrix{
f \acts X\ar@{-->}[r]^{\phi} &Y\racts g
}$$
where $\phi$ is dominant and generically finite.
Then $f$ is weakly primitive (resp.~has a Zariski dense orbit) if and only if so is $g$.
\end{lemma}

In general, the product case for Conjecture ZDO is hard. We show a special case below.
\begin{lemma}\label{lem-zdo-prod}
Let $g:Y\to Y$ and $h:Z\to Z$ be surjective endomorphisms of projective varieties.
Suppose $g$ and $h$ have Zariski dense orbits, $\deg g<\deg h$ and $\dim Z=1$.
Then $g\times h$ has a Zariski dense orbit.
\end{lemma}
\begin{proof}
Let $X:=Y\times Z$ and $f:=g\times h$.
Let $y\in Y$ and $z\in Z$ such that $\overline{O_g(y)}=Y$ and $\overline{O_h(z)}=Z$.
Then $\overline{O_{g\times h}(y,z)}$ dominates both $Y$ and $Z$.
By \cite[Lemma 2.7]{MMSZ23}, we may assume $P:=\overline{O_{g\times h}(y,z)}$ is irreducible and $f$-invariant after suitable iteration and replacing $(y,z)$ by $(g(y), h(z))$.
If $P\neq X$, then $P\to Y$ is generically finite and hence 
$$\deg g<\deg h\le \deg f|_P=\deg g,$$
a contradiction.
So $P=X$ and the lemma is proved. 
\end{proof}

\begin{lemma}\label{lem-zdo-reldim1}
Consider the following equivariant dynamical systems of normal projective varieties
$$\xymatrix{
f \acts X\ar@{-->}[r]^{\phi} &Y\racts g
}$$
satisfying the following:
\begin{enumerate}
\item $\phi$ is dominant,
\item $X$ is $\Q$-factorial and $\dim X=\dim Y+1$,
\item $\deg f>\deg g=1$, and
\item $g$ has a Zariski dense orbit.
\end{enumerate}
Then $f$ has a Zariski dense orbit.
\end{lemma}
\begin{proof}
Note that Conjecture ZDO holds for curves (cf.~\cite[Corollary 9]{Ame11}).
So we may assume $\dim Y>0$.
Suppose the contrary that $f$ has no Zariski dense orbit.

We claim that any $f$-periodic prime divisor $P$ of $X$ dominating $Y$ is $f^{-1}$-periodic.
After iteration, we may assume $f(P)=P$.
Suppose the contrary that $P$ is not $f^{-1}$-periodic.
Then we can find an infinite sequence of prime divisors $P_i$ with $i\ge 0$ such that $f(P_{i+1})=P_i$, $P_i\neq P_j$ for any $i\neq j$, and $P_0=P$.
Note that $\deg f|_P=\deg g=1$ by the assumption (3).
By Proposition \ref{prop-f*f*}, we have $f^*P\sim_{\Q} (\deg f) P$, and hence
$f^*P_i\sim_{\Q} (\deg f) P_i$ for all $i$.
Suppose $P_i\cap P_j\neq\emptyset$ for some $i>j\ge 0$.
Then $P_{i-j}\cap P\neq \emptyset$ and hence $P_{i-j}|_P$ is a non-zero effective divisor on $P$.
Since $f|_P$ is invertible, for any eigenvalue $\lambda$ of $(f|_P)^*|_{\NS(P)}$, $1/\lambda$ is an algebriac integer.
In particular, $(f|_P)^*|_{\NS(P)}$ has no rational eigenvalue $>1$. 
However, $(f|_P)^*(P_{i-j}|_P)=(\deg f) P_{i-j}|_P$, a contradiction.
So $P_i\cap P_j=\emptyset$ for any $i\neq j$.
By \cite[Theorem 1.1]{BPS16}, $\kappa(X, P)=1$.
Let $\psi:X\dashrightarrow Z$ be the $f$-Iitaka fibration of $P$.
Then $\dim Z=1$ and the induced map $X\dashrightarrow Y\times Z$ is generically finite and dominant.
By Theorem \ref{thm-fiitaka-polarized}, $f|_Z$ is $(\deg f)$-polarized and hence $f|_Z$ has a Zariski dense orbit (cf.~\cite[Corollary 9]{Ame11}).
By Lemmas \ref{lem-zdo-gf}, \ref{lem-zdo-prod} and the assumption (4), $f$ has a Zariski dense orbit, a contradiction.
The claim is proved.

Let $y\in Y$ such that $O_g(y)$ is Zariski dense in $Y$.
Note that $\dim X=\dim Y+1$ and $X$ is normal.
Then $\phi^{-1}(y)$, the set of well-defined points of $\phi$ mapping to $y$, is non-empty and consists of infinitely many points.
Let $x\in \phi^{-1}(y)$.
Then $\overline{O_f(x)}$ dominates $Y$, is not equal to $X$, and contains an $f$-periodic prime divisor $Q_1$ dominating $Y$ (cf.~\cite[Lemma 2.7]{MMSZ23}).
By the previous claim, $Q_1$ is $f^{-1}$-periodic.
Then $\phi^{-1}(y)\backslash \bigcup\limits_{i=0}^{+\infty} f^{-i}(Q_1)\neq \emptyset$ after replacing $y$ by $g^s(y)$ for some $s>0$.
Let $x_2\in \phi^{-1}(y)\backslash \bigcup\limits_{i=0}^{+\infty} f^{-i}(Q_1)$.
We have $O_f(x_2)\cap Q_1=\emptyset$.
So there is another $f$-periodic prime divisor $Q_2$ dominating $Y$.
Repeating this process, we obtain an infinite sequence of different $f^{-1}$-periodic prime divisors $Q_i$ dominating $Y$.
Let $\Sigma_0$ be the union of $f^{-1}$-periodic prime divisors in $\Supp R_f$ and $\Sigma:=\bigcup\limits_{j=0}^{+\infty} f^{-j}(\Sigma_0)$.
Then $\Sigma$ is still a divisor.
If $Q_i\subseteq \Supp R_{f^s}$ for some $s>0$, then $Q_i\subseteq \Sigma$.
So when $i\gg 1$, we have $Q_i\not\subseteq \Supp R_{f^s}$ for any $s>0$.
In particular, we have $f^*Q_i=Q_i$ for some $i$ after iteration.
Then $\deg g=\deg f|_{Q_i}=\deg f>1$, a contradiction.
\end{proof}

Now we prove the birational EMMP: the ZDO version (compare with Theorem \ref{mainthm-A}).

\begin{proof}[Proof of Theorem \ref{mainthm-D}]
Since $f$ is weakly primitive, $\kappa(X)\le 0$.
Note that Conjecture ZDO holds when $f$ is quasi-abelian (cf.~\cite[Theorem 1.2]{GS17}).

Suppose $\kappa(X)=0$. 
We recall the proof of Theorem \ref{mainthm-B}.
After an $f$-equivariant \'etale cover, $X\cong S\times C$ where $S$ is a K3 surface and $C$ is an elliptic curve.
Further, $f=g\times h$ with $g:S\to S$ an automorphism and $h:C\to C$ of $\deg h>1$.
Since $f$ is weakly primitive, so are $g$ and $h$.
Since Conjecture ZDO holds for projective curves and surfaces (cf.~\cite[Corollary 9]{Ame11} and \cite[Theorem 1.9]{JXZ23}),
$g$ and $h$ have Zariski dense orbits.
By Lemma \ref{lem-zdo-prod}, $f$ has a Zariski dense orbit, a contradiction.
So $X$ is uniruled.

Suppose the contrary that there is a birational MMP
$$\xymatrix{
X=X_1\ar[r]^{\pi_1}& X_2\ar@{-->}[r]^{\pi_2} &\cdots\ar@{-->}[r]^{\pi_{n-1}} &X_{n}\ar@{-->}[r]^{\pi_{n}} &X_{n+1}
}$$
such that $\pi_i$ is $f_i:=f|_{X_i}$-equivariant after iteration for $i<n$ and $\pi_n$ is not $f_n$-equivariant even after any iteration.
By Theorem \ref{thm-emmp-flip}, $\pi_n$ is a divisorial contraction.
Let $E$ be the strict transform of the exceptional divisor of $\pi_n$ in $X$.
Note that $E$ is not $f^{-1}$-periodic by \cite[Lemma 6.2]{MZ18} and $\kappa(X,E)=0$.
Let $\phi:X\dashrightarrow Y$ be the $f$-Iitaka fibration of $E$.
By Theorem \ref{thm-fiitaka2}, $\dim Y=2$ and $\deg f|_Y=1$.
Since $f$ is weakly primitive, so is $g$.
Note that ZDO holds for projective surfaces by \cite[Theorem 1.9]{JXZ23}.
So $g$ has a Zariski dense orbit.
Then we get a contradiction to Lemma \ref{lem-zdo-reldim1}. 
\end{proof}

\end{document}